\RequirePackage{fix-cm}
\documentclass[smallextended]{svjour3}       
\smartqed  
\usepackage{graphicx}
\usepackage{bm}
\usepackage{pstricks}
\usepackage{rotating}
\usepackage[numbers]{natbib} 
\usepackage[utf8]{inputenc}
\usepackage{url}
\usepackage{amssymb,latexsym,verbatim,amsmath}
\usepackage{graphicx,epsfig,amssymb,color}

\usepackage[normalem]{ulem}
\usepackage{cancel}

\newcommand{\hlambda}{\hat{\lambda}}
\newcommand{\htau}{\hat{\tau}}
\newcommand{\htheta}{\hat{\theta}}
\newcommand{\R}{\mathbb{R}}
\def\calL{\mathcal{L}}
\def\calD{\mathcal{D}}
\def\calE{\mathcal{E}}

\def\eps{\varepsilon}
\def\calL{\mathcal{L}}
\def\calN{\mathcal{N}}
\def\calO{\mathcal{O}}
\def\A{A}
\def\B{B}
\def\C{S}
\def\b{\underline{b}}

\def\cc{\widetilde{c}}
\def\g{{\tt g}}
\def\Zsf{Z^{\rm sf}}

\def\tE{\widetilde{\mathcal{E}}}

\def\ta{\widetilde{a}}
\def\te{\widetilde{e}}

\usepackage{todonotes}

 \journalname{PREPRINT}
\begin{document}

\title{Unfolding symmetric Bogdanov-Takens bifurcations for front dynamics in a reaction-diffusion system\thanks{HI was partially supported by JSPS KAKENHI Grant Number JP15K04885.
} \thanks{PvH was supported under the Australian Research Council’s Discovery Early Career Researcher Award funding scheme DE140100741.}\thanks{JR was supported in part by DFG grant Ra 2788/1-1 and TRR 181 project number  274762653.}}

\titlerunning{Symmetric Bogdanov-Takens bifurcations in a reaction-diffusion system}       

\author{M. Chirilus-Bruckner \and P. van Heijster \and H. Ikeda \and J.D.M.\ Rademacher
}

\institute{M. Chirilus-Bruckner \at
              Mathematisch Instituut, Leiden University, P.O. Box 9512, 2300 RA Leiden, The Netherlands \\
              \email{m.chirilus-bruckner@math.leidenuniv.nl}          
           \and
           P. van Heijster \at
              Mathematical Sciences School, Queensland University of Technology, GPO Box 2434, Brisbane, QLD 4001, Australia
                 \and 
              H. Ikeda \at Department of Mathematics, University of Toyama, Gofuku 3190, Toyama, 930-8555, Japan
              \and
              J.D.M. Rademacher \at
              Fachbereich Mathematik, Universit\"at Bremen, Postfach 22 04 40, 20359 Bremen, Germany
}

\date{January 2019}

\maketitle

\begin{abstract}
This manuscript extends the analysis of a much studied singularly perturbed three-component reaction-diffusion system for front dynamics in the regime where the essential spectrum is close to the origin. We confirm a conjecture from 
a preceding paper by proving that the triple multiplicity of the zero eigenvalue gives a Jordan chain of length three. Moreover, we simplify the center manifold reduction and computation of the normal form coefficients by using the Evans function for the eigenvalues. Finally, we prove the unfolding of a Bogdanov-Takens bifurcation with symmetry in the model. This leads to stable periodic front motion, including stable traveling breathers, and these results are illustrated by numerical computations.
\end{abstract}

\keywords{Three-component reaction--diffusion system \and Front solution \and Singular perturbation theory \and Evans function \and Center manifold reduction \and Normal forms}

\section{Introduction} 
Localized structures, such as fronts, pulses, stripes and spots are close to their trivial background states in large regions of their spatial domain and, in small regions, transition between trivial background states, or make an excursion away and back from one of them. These localized structures 
often form the backbone of more complex patterns in reaction-diffusion equations \cite{HM94, NU01,P93}. Understanding localized structures is thus a crucial step towards understanding complex patterns. While significant progress has been made over the past few decades to understand such localized structures, see \cite[e.g.]{BELL,D03,D07,EI,K06, P02,R13,S02,S05} and references therein, many open questions remain.

One of these concerns the influence of the essential spectrum when it approaches the imaginary axis, and the so-called spectral gap becomes asymptotically small. This question is the main motivation of the current manuscript, which is a continuation of \cite{CDHR15}. Here the localised structures are fronts, which are singular perturbations of sharp interfaces of the Allen-Cahn equation coupled to linear large scale fields, and we view this as a caricature model for multiscale effects on interfacial dynamics and energy transfer. There is a large body of literature on related models and also planar fronts with a more physical perspective, e.g., \cite{MBHT01,M15} and the references therein.

We take a mathematical viewpoint and consider the three-component singularly perturbed reaction-diffusion system
\begin{align} \label{eq:three_component_system}
 \left\{ 
 \begin{array}{rlrl}
                                         \partial_t U & = &  \eps^2 \partial_x^2 U & + \ U - U^3 - \eps(\alpha V + \beta W + \gamma) \, , \\[.2cm]
  \frac{\hat{\tau}}{\eps^2}   \,  \partial_t V  & = &    \partial_x^2 V & + \ U - V \, , \\[.2cm]
  \frac{\hat{\theta}}{\eps^2} \,  \partial_t W  & = & D^2\partial_x^2 W & + \ U - W \, ,
 \end{array}
 \right.
\end{align}
with $ x  \in \mathbb{R}, t \geq 0, U = U(x,t), V = V(x,t), W = W(x,t) \in \mathbb{R}$, and parameters $ \alpha, \beta, \gamma \in \mathbb{R}, \hat{\tau}, \hat{\theta} > 0, D > 1$\footnote{The condition $D>1$ implies that the $W$-component has the largest diffusion coefficient and its profile thus changes the slowest (as function of the spatial variable $x$), see for instance Figure~\ref{fig:front}. This condition stems from the original gas-discharge system and is not a mathematically necessary requirement, though convenient. For $D<1$ the $W$-component and the $V$-component simply interchange roles.}, as well as the singular perturbation scale $  0 < \eps \ll 1$ such that all parameters, including $\hat{\tau}, \hat{\theta}$, are $\mathcal{O}(1)$ with respect to $\eps$. 
A dimensional version of this system was introduced in the mid-nineties to study gas-discharge systems on a phenomenological level \cite[eg]{PUR,OR,SCHENK}. Afterwards, versions of \eqref{eq:three_component_system} have been studied extensively by mathematicians since it supports localized solutions that undergo complex dynamics while the model is still amendable for rigorous analysis \cite[e.g]{CHEN,CDHR15,DHK09,N03,N07,HC,HC2,H2,H3,H4,HS11,HS14,VANAG}. 
In the predecessor paper \cite{CDHR15}, it was shown that system \eqref{eq:three_component_system} supports uniformly traveling front solutions $ (U,V,W)(x,t) = (u^{\rm tf}, v^{\rm tf}, w^{\rm tf})(x-\eps^2 c t) $ that transition from the background state near $(-1,-1,-1) $ to the background state near $(1,1,1)$
if the system parameters and the velocity $ c $ satisfy the existence condition
\begin{align}\label{eq:existence_condition}
 \Gamma_\eps(c) =  \alpha  \frac{c \htau}{\sqrt{c^2\htau^2 + 4}}  + \beta  \frac{c \htheta}{D\sqrt{c^2\frac{\htheta^2}{D^2} + 4}}  + \gamma - \frac{\sqrt{2}}{3} c + h.o.t. = 0;
\end{align}
here and below `h.o.t' stands for `higher order terms', see \cite{CDHR15}.

This trivially yields $\gamma$ as a function of the remaining parameters and $c$, however, for the partial differential equation (PDE) dynamics it is decisive to view $\Gamma$ as a function of the auxiliary velocity parameter $c$. For $\gamma=0$ the existence condition is an odd function of $c$, and for $\gamma \neq 0$ parameters can always be adjusted to find a stationary front solution with $c=0$. Having in mind the symmetry breaking nature of $\gamma$ in \eqref{eq:three_component_system}, we will focus only on $\gamma=0$ for the analysis of this manuscript. From the viewpoint of the Allen-Cahn energy, $\gamma\neq0$ is a nontrivial external energy flux so that traveling fronts with nonzero velocity for $\gamma=0$ are somewhat surprising, cf.\ e.g.\ \cite{IMN,NMIF90}.
Typically, the energy flux $\gamma\neq 0$ is transferred to interface motions, which we find can also be oscillatory due to the coupled fields. 
It is well known that these stationary front solutions can undergo stationary bifurcations and the full analysis of the bifurcation structure in \cite{CDHR15} yields a (partially unfolded) butterfly catastrophe.
\begin{figure}
\begin{center}
\begin{tabular}{cc}
\includegraphics[width=0.4\textwidth]{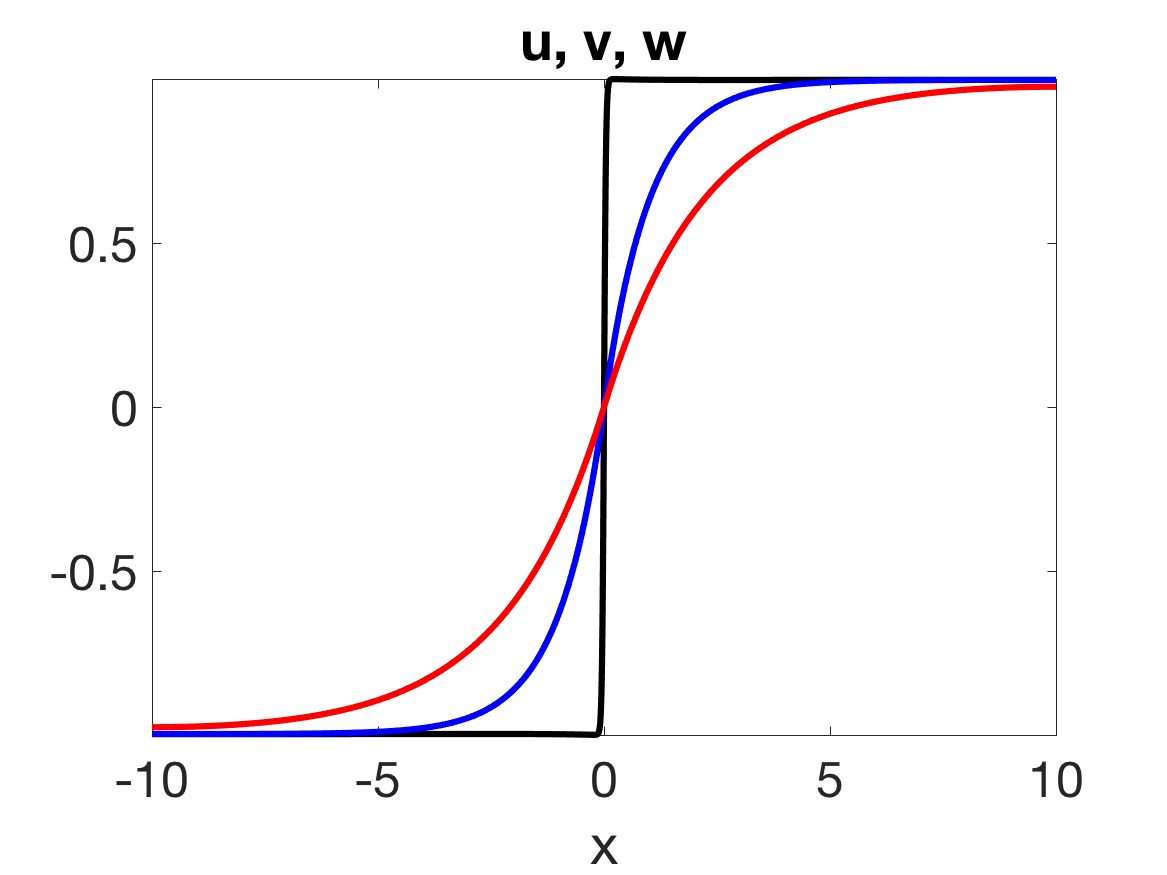}
&\includegraphics[width=0.4\textwidth]{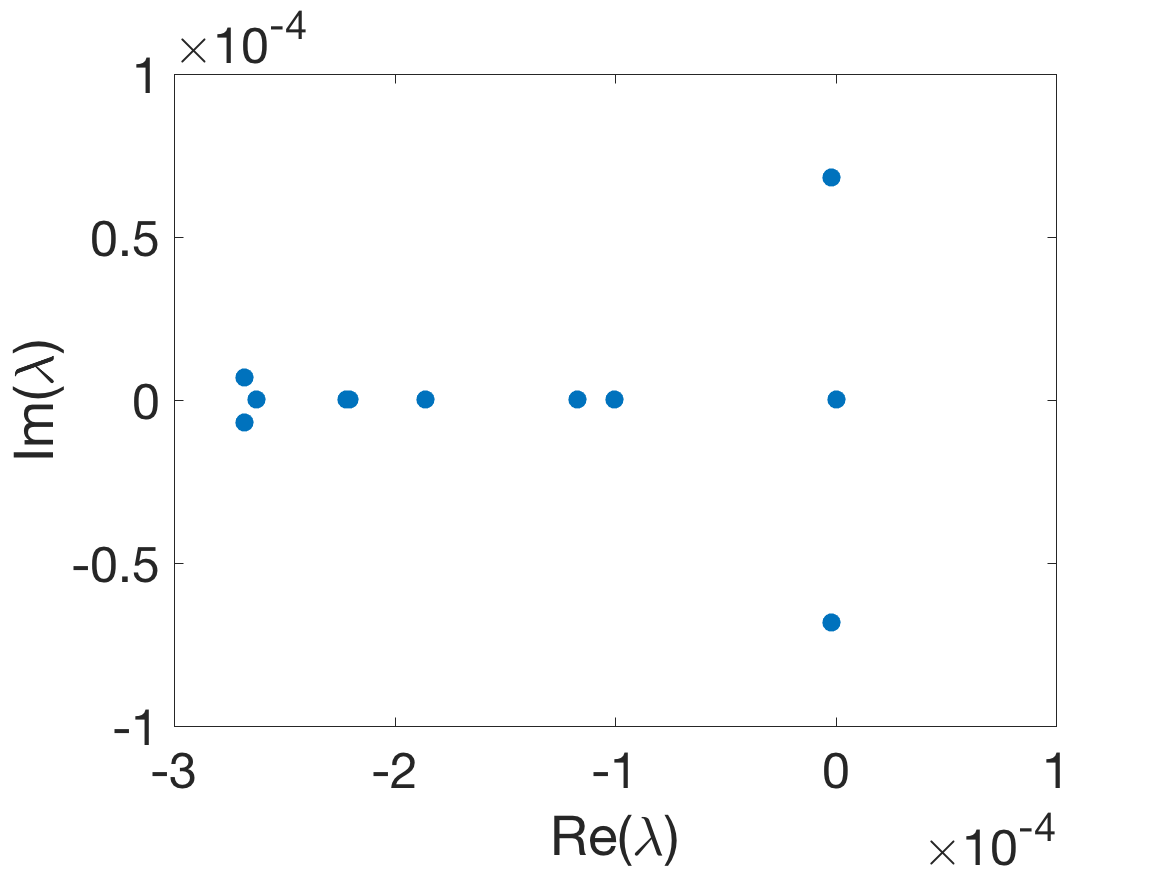}\\
(a) & (b)
\end{tabular}
\end{center}
\caption{(a) Sample profile of a numerically computed stationary front solution near a Hopf bifurcation with the profiles of $u$ (black) $v$ (blue) and $w$ (red); parameter values are those of Fig.~\ref{f:triple-in-g2} with $g_2\approx-0.003$. (b) Spectrum of this front with gap between the leading three eigenvalues as theoretically predicted.}\label{fig:front}
\end{figure}

Using geometric singular perturbation theory, the stationary front solutions to \eqref{eq:three_component_system} can be specified to leading order in the perturbation parameter $ \eps $  as
\begin{align}\label{eq:front}
\hspace{-.5cm}
\begin{array}{c}
\left[
\begin{array}{c}
U(x,t) \\[.2cm]
V(x,t) \\[.2cm]
W(x,t) 
\end{array}
\right]
\end{array}
=
\begin{array}{c}
\left[
\begin{array}{c}
u^{\rm h}(x) \\[.2cm]
v^{\rm h}(x) \\[.2cm]
w^{\rm h}(x) 
\end{array}
\right]
\end{array}
=
\begin{array}{c}
\left[
\begin{array}{c}
u_0^{\rm h}(x) \\[.2cm]
v_0^{\rm h}(x) \\[.2cm]
w_0^{\rm h}(x) 
\end{array}
\right]
\end{array}
+ h.o.t.,
\end{align}
with
\begin{align*}
\hspace{-.5cm}
\begin{array}{c}
\left[
\begin{array}{c}
u_0^{\rm h}(x) \\[.2cm]
v_0^{\rm h}(x) \\[.2cm]
w_0^{\rm h}(x) 
\end{array}
\right]
= 
\left[
\begin{array}{c}
{\rm tanh}\left[\frac{x}{\sqrt{2}\eps}\right] \\[.2cm]
0\\[.2cm]
0
\end{array}
\right]
\chi_{f}(x)
+
\displaystyle\sum_{\sigma\in\{+,-\}}
\sigma\left[
\begin{array}{c}
1 \\[.2cm]
1-e^{x} \\[.2cm]
1-e^{x/D} 
\end{array}
\right]
\chi_{s\sigma}(x)
\end{array}
\end{align*}
and where the slow/large scale and fast/small scale behavior has been captured through
\[
\chi_{s-} = \chi_{\left(- \infty, -{\sqrt{\eps}}\right)}\, , \quad  \chi_{f} = \chi_{\left[-{\sqrt{\eps}},+{\sqrt{\eps}}  \right]}\, , \quad  \chi_{s+} = \chi_{\left(+{\sqrt{\eps}}, +\infty \right)}.
\]

It has been shown in \cite{CDHR15} that the operator arising from the linearization of \eqref{eq:three_component_system} around a stationary front has the following spectral properties, also illustrated in  Figure~\ref{fig:front}(b): 
First, its essential spectrum is located in a sector of the left half plane and bounded away from the imaginary axis by $ \max\{-2, -\eps^2/\hat{\tau}, - \eps^2/\hat{\theta}\} $. 
Second, the only point spectrum that could lead to instabilities are small eigenvalues $ \lambda = \eps^2 \hat{\lambda} $. As usual for translation symmetric PDE, one such eigenvalue is $ \lambda = 0 $ with eigenfunction being 
the spatial derivative of the stationary front.
Third, the algebraic multiplicity of $ \lambda = 0 $ can only be one, two or three, see also the upcoming Proposition~\ref{proposition:evans_function}\footnote{Two of these eigenvalues have emerged from the essential spectrum upon increasing $\htau$ and/or $\htheta$ from $\mathcal{O}(\eps^{2})$ to $\mathcal{O}(1)$, see \cite{H2}.}.

In \cite{CDHR15} the nonlinear stability analysis and bifurcations of stationary fronts has been treated for the special case of unfolding around a double zero eigenvalue. The more challenging case of unfolding the triple zero was left as an open problem in \cite{CDHR15} and is the goal of the present manuscript. 
We will use center manifold analysis in the vicinity of a triple zero eigenvalue to derive the dynamics of pseudo-front solutions with non-uniform speed $c = c(t) $. 
Whilst the use of center manifold reduction is by now standard for instabilities caused by point spectrum, the main novelties of the article are as follows.

First,
although the algebraic multiplicity three of the zero eigenvalue can easily be read off from the Evans function, see Proposition~\ref{proposition:evans_function}, the corresponding eigenspace needs more analysis. Formal computations, as demonstrated in Appendix~\ref{app:formal_computation}, suggest a Jordan block of length three arises and, hence, that there are two generalized eigenfunctions. 
We confirm this by an abstract rigorous argument for the existence of generalized eigenfunctions.  
Generally, there are two different methods for solving the singularly perturbed
linearized eigenvalue problem: an analytical approach called the
Singular Limit Eigenvalue Problem (SLEP) method \cite{NF87, NMIF90} and a geometrical
approach called the Nonlocal Eigenvalue Problem (NLEP) method \cite{DGK}.
Although both methods are based on the linearized stability principle, the
former method solves the linearized eigenvalue problem directly and
derives a well-defined singular limit equation called the SLEP equation
as $\varepsilon \to 0$, while the latter method defines the Evans function \cite{AGJ}
for the linearized equations and subsequently applies a topological method to it.
The SLEP-method gives very detailed information on the behavior of the
critical eigenvalues for small $\varepsilon$, whereas the NLEP-method
can be applied to wider class of equations. Here, we use the SLEP-method to find generalized eigenfunctions corresponding to the
triple zero eigenvalue. This is slightly different from
a usual eigenvalue problem because the zero eigenvalue has been determined previously,  but it is the same in spirit:
we find the relation between the system parameters included in the original
eigenvalue problem and an eigenfunction, and this relation corresponds to an eigenvalue. This is also the crux
to finding generalized eigenfunctions, and these relations play, in essence, the role of solvability
conditions.
In fact, we expect our results can be further generalized to extensions of \eqref{eq:three_component_system} that lead to Jordan chains of arbitrary length, see \S\ref{s:conclusion}. 

Second, and more relevant for analysing the concrete PDE dynamics, we circumvent the straightforward but tedious computation of normal form coefficients of the usual center manifold reduction procedure by using the information on existence and stability of uniformly traveling fronts,  
a strategy that we believe is of interest beyond our setting. 

As a result, we obtain a reduced equation featuring a symmetric Bogdanov-Takens bifurcation scenario, and we prove its unfolding by the system parameters. Specifically, we prove that the front positions $a(t)$ satisfy an ODE of the form
\begin{equation*}
\frac{d^3}{dt^3}a = \eps^6 G(\dot a, \ddot a,\mu),
\end{equation*}
where $\mu$ combines system parameters used for unfolding the bifurcations. In case of the symmetric Bogdanov-Takens point, the subsystem for the velocities $c$, with $\eps^2 c= \dot a$, on the slow time scale $'=\eps^{-2} d/dt$ has the normal form
\begin{equation*}
c'' = g_1(\mu)c + g_{30} c^3 + c'\left(g_2(\mu) + g_{40}c^2\right)
\end{equation*}
and we give explicit formulas for $g_{30}  g_{40}\neq 0$ and the relevant expansion of $g_1, g_2$.
In particular, the unfolding generates various forms of periodic front motion.  These results are  illustrated by numerical computations as in Figure~\ref{f:numintro}, see \S\ref{s:num}. 

\begin{figure}[tbp]
\begin{center}
\begin{tabular}{ccc}
\includegraphics[width=0.3\textwidth]{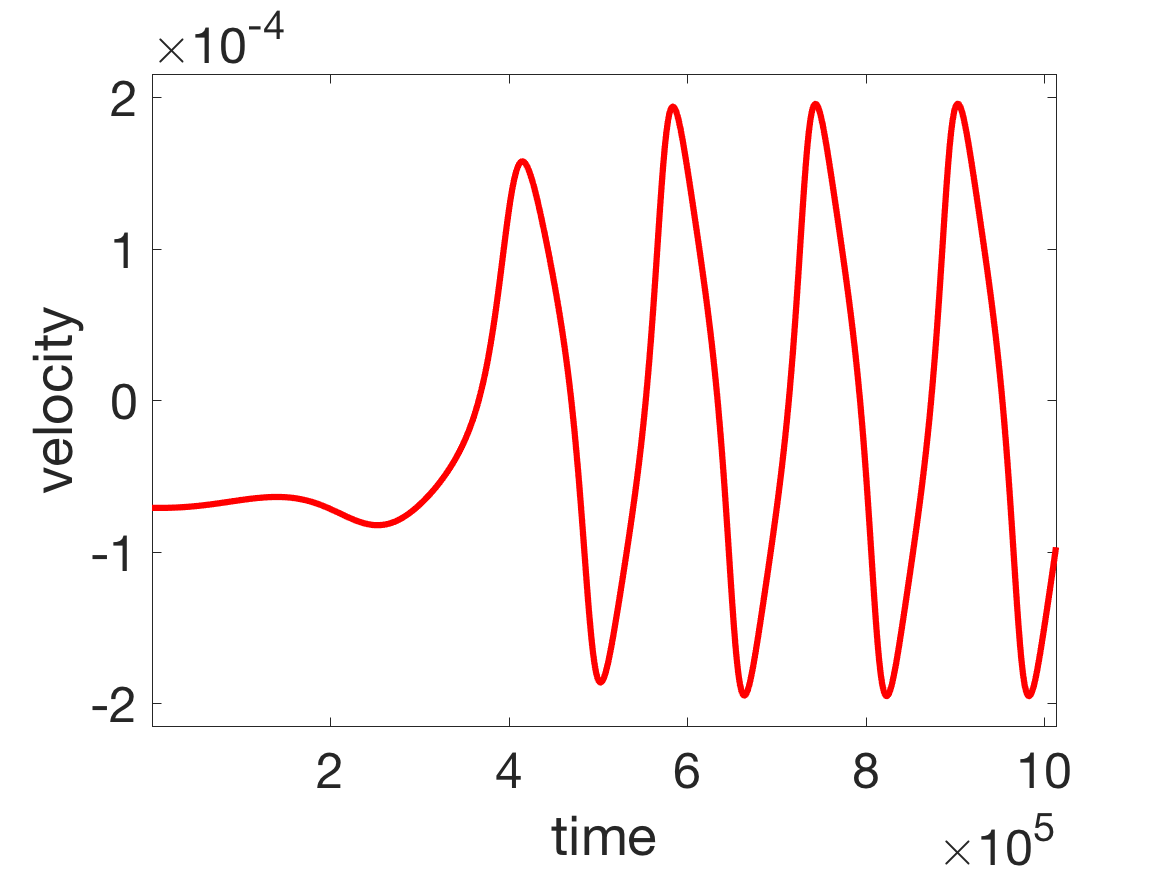}
&\includegraphics[width=0.3\textwidth]{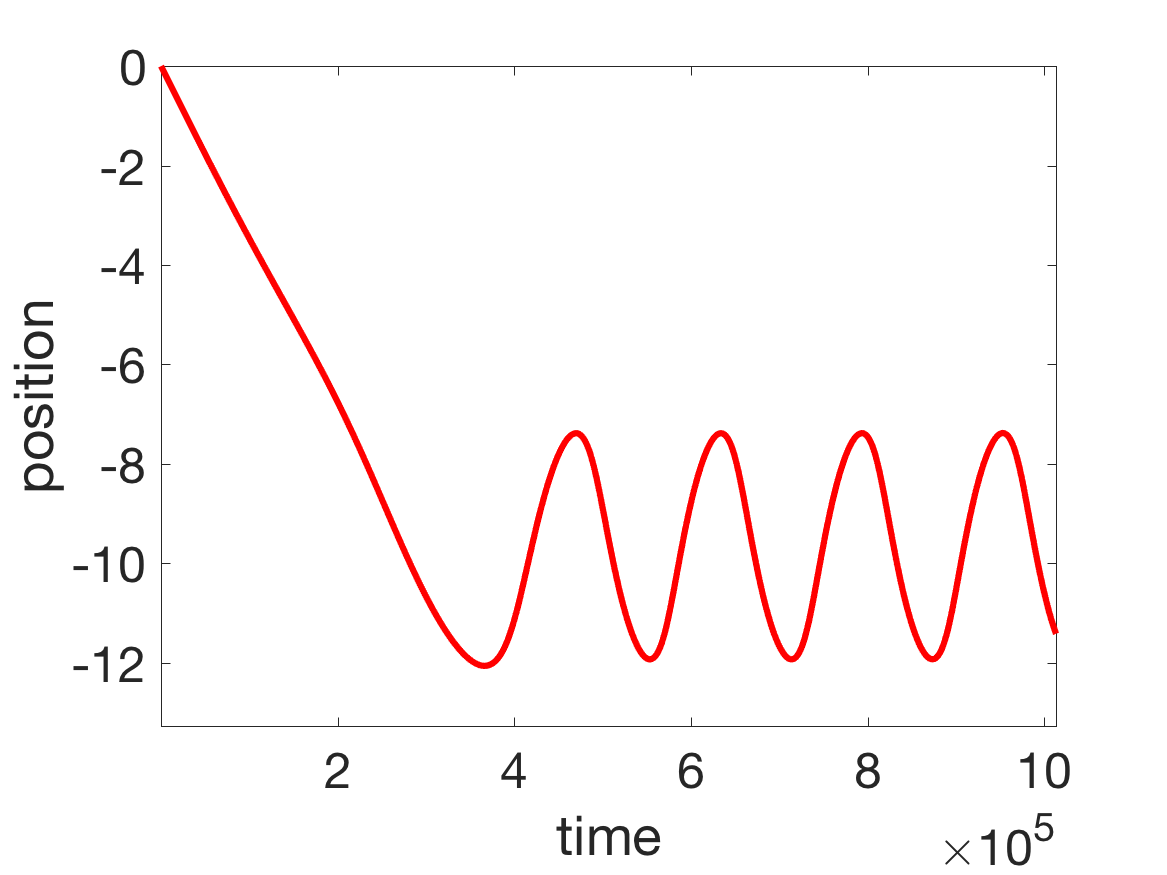}
&\includegraphics[width=0.3\textwidth]{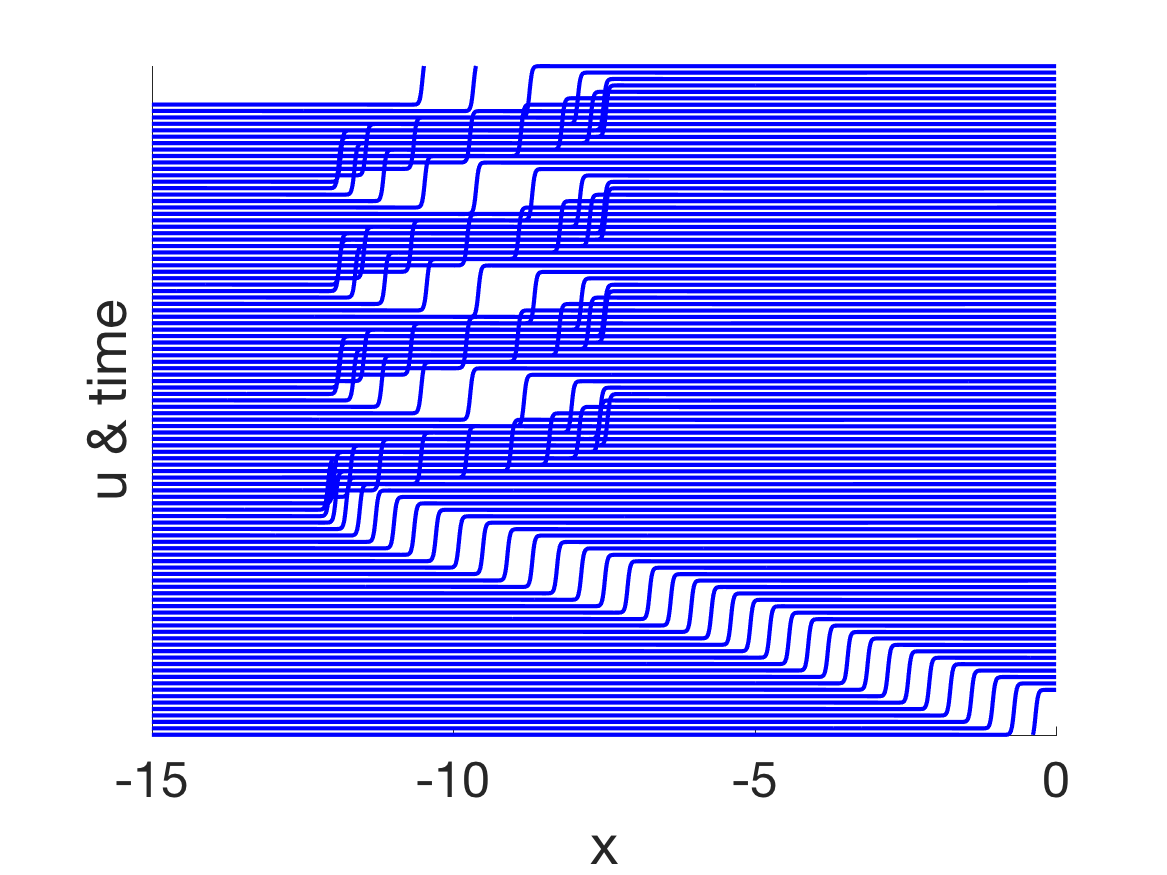}\\
(a) & (b) & (c)\\
\includegraphics[width=0.3\textwidth]{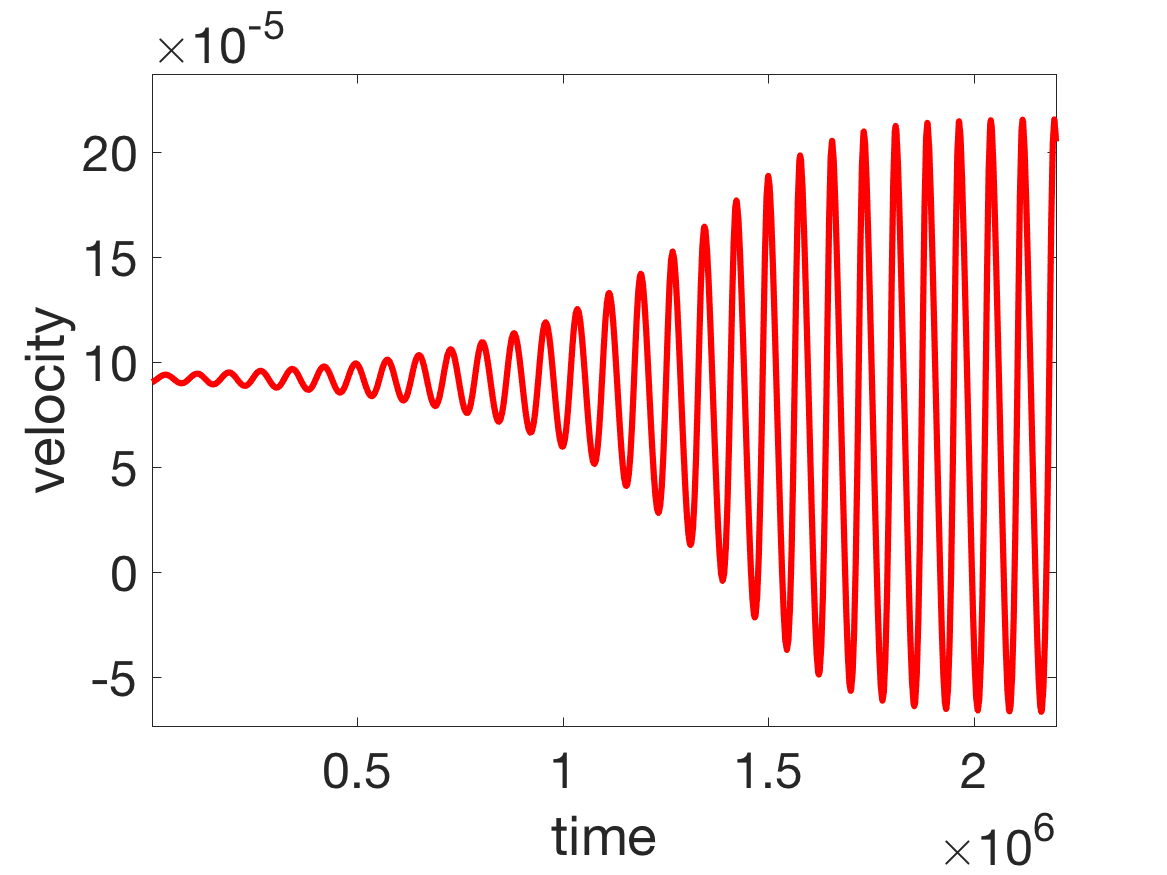}
&\includegraphics[width=0.3\textwidth]{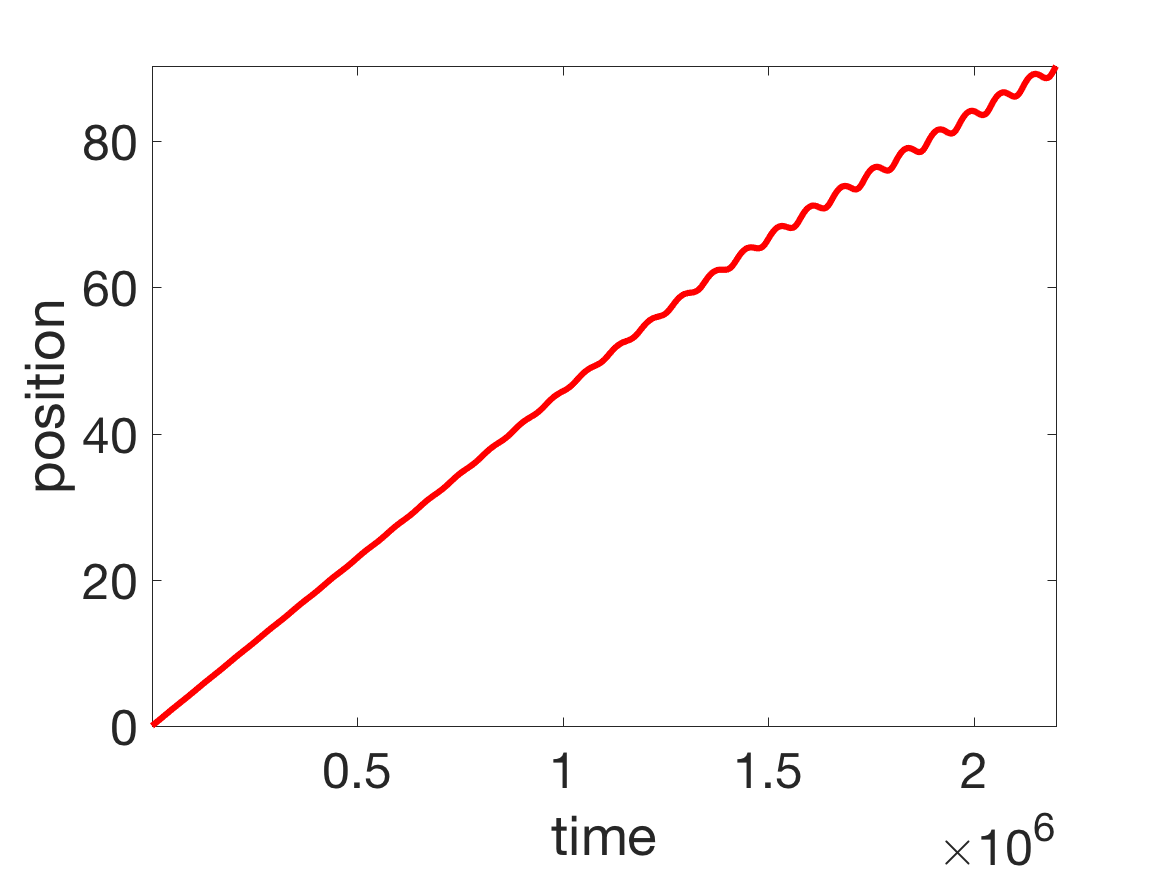}
&\includegraphics[width=0.3\textwidth]{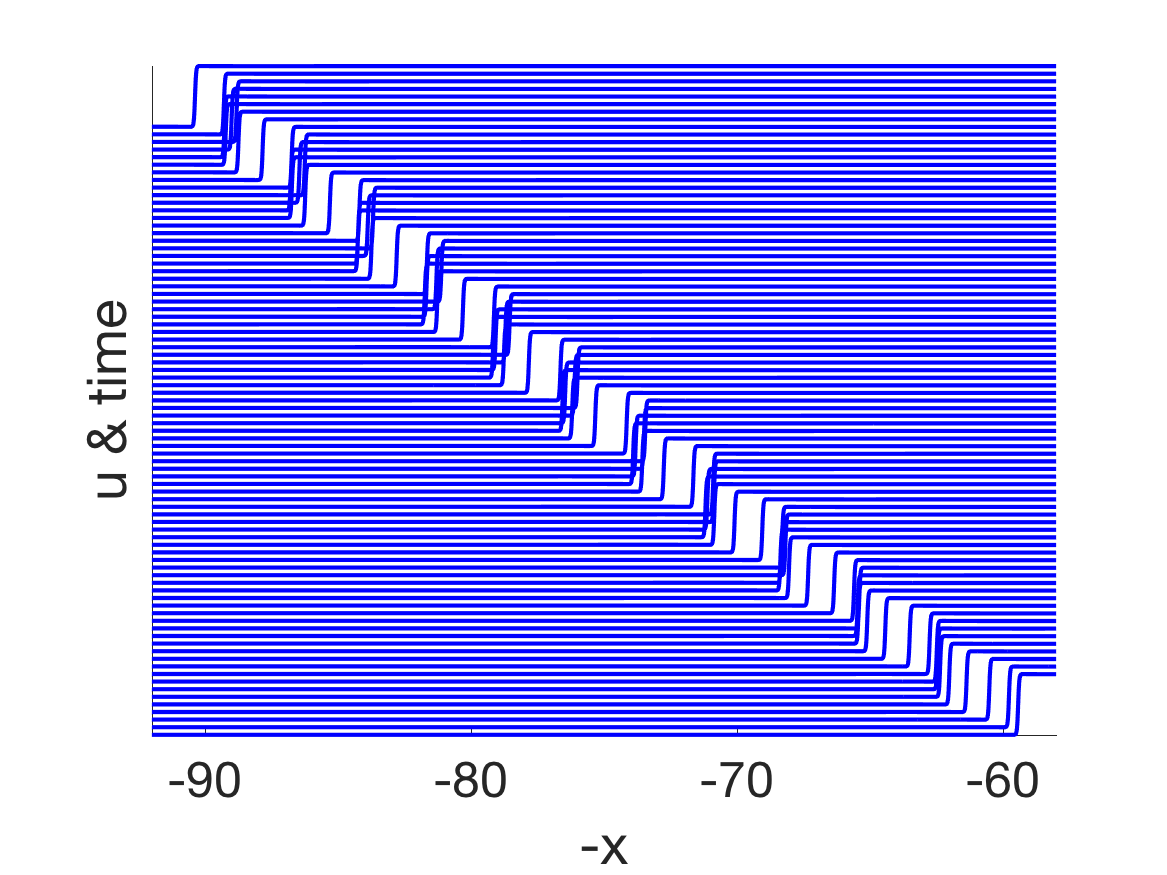}\\
(d) & (e) & (f)
\end{tabular}
\caption{Plots of velocites (a,d) and two views of positions of pseudo-fronts obtained from numerical simulations of \eqref{eq:three_component_system} for parameters located by using the center manifold analysis. The solutions correspond to heteroclinic orbits from an equilibrium to a periodic orbit in the center manifold, see \S\ref{s:num} for details. (a-c): a perturbation of an unstable stationary front leads to a periodic front motion in both velocity and position, see Figure~\ref{f:triple-in-g1-g2pos}. (d-f): non-periodic positions, i.e., a `traveling breather', occurs due to $\gamma\neq 0$, see Figure~\ref{f:movper}. (f) shows a subset from (e) and reflected for better view.\label{f:numintro}}
\end{center}
\end{figure}

This manuscript is organized as follows. In \S\ref{sec:stability_stationary_front} we discuss the results from \cite{CDHR15} and show that the operator arising from linearization around a stationary front of \eqref{eq:three_component_system} possesses a Jordan chain of length three and we compute, to leading order, the second generalized eigenfunction.
In \S\ref{sec:cmr} we use center manifold analysis in the vicinity of a triple zero eigenvalue to derive, and subsequently study, the dynamics of pseudo-front solutions with non-uniform speed $c = c(t) $. We end the manuscript with a discussion of potential directions for future work.


\section{Stability and eigenfunctions of stationary front solutions}\label{sec:stability_stationary_front}
In order to state results on the stability of stationary fronts, it is convenient to write our system \eqref{eq:three_component_system} in the more concise form
\begin{align*}
  M(\htau, \htheta) \, \partial_t Z = F(Z;\alpha,\beta, D) \, ,
\end{align*}
where $M,F$ are, with explicit $\eps$-dependence, given by
\begin{align}\label{eq:M_F}
\begin{aligned}
  M(\htau, \htheta;\eps) &=  \left(\begin{array}{ccc} 1 & 0 & 0 \\ 0 & \frac{\htau}{\eps^2} & 0 \\ 0 & 0 & \frac{\htheta}{\eps^2} \end{array} \right) \, , \quad {\textnormal{and}} 
  \qquad \qquad \qquad \qquad
  \\
  F(Z;\alpha,\beta, D;\eps) &=  \left(\begin{array}{c} \eps^2 \partial_x^2 Z^u + Z^u - (Z^u)^3 - \eps(\alpha Z^v + \beta Z^w) \\[.2cm] \partial_x^2 Z^v - Z^v + Z^u \\[.2cm] D^2 \partial_x^2 Z^w - Z^w + Z^u  \end{array} \right) \, .
\end{aligned}
\end{align}
Linearization around the stationary front $ \Zsf = (u^{\rm h}, v^{\rm h}, w^{\rm h}) $ from \eqref{eq:front} gives rise to the eigenvalue problem
\begin{align}
\nonumber
  \lambda \,  M(\htau, \htheta) \, \Phi  = \partial_Z F(\Zsf;\alpha,\beta, D) \, \Phi \,,
\end{align}
so, in the following, we will be interested in the spectrum of the operator
\begin{align}\label{eq:L}
  \calL:= M(\htau, \htheta)^{-1} \partial_ZF(\Zsf;\alpha,\beta, D) \, .
\end{align}

Various results on the critical eigenvalues and the corresponding eigenfunctions were obtained in Lemmas 5-8 and Corollary 3 from \cite{CDHR15}, which we reformulate next.
As we will see, unfolding the bifurcations can be realised with $ \alpha, \beta, D $, which is based on certain normal form coefficients introduced later. It is, however, instructive to first consider the quantities
\begin{align}\label{eq:kappas}
 \kappa^0_1 := \alpha \htau + \beta \frac{\htheta}{D} - \frac{2 \sqrt2}{3} \, , \quad  \kappa^0_2 :=  \alpha \htau^2 + \frac{\beta}{D} \htheta^2 \, , \quad \kappa^0_3 := \alpha \htau^3 + \beta \frac{\htheta^3}{D^3},
\end{align}
which already appeared \cite{CDHR15} and where the upper index $0$ refers to $\eps=0$, the limit that forms the backbone of all our computations. Statements in terms of $\kappa^0_j$, $j=1,2,3$ thus implicitly refer to the parameters $ \alpha, \beta, D, \htau, \htheta $.


\begin{proposition}[\bf Stability of stationary front solutions \cite{CDHR15}]\label{proposition:evans_function}
Let $ \eps > 0 $ be chosen sufficiently small. The critical spectrum of $\calL$ from \eqref{eq:L} on $L^2(\R)$ with domain $H^2(\R)$, or $C^0_{\rm unif}(\R)$ with domain $C^2_{\rm unif}(\R)$, consists of at most three small eigenvalues  $ \lambda = \eps^2 \hat{\lambda} + o(\eps^2)$ given by the roots of the Evans function
\begin{align}
\nonumber
 {\calD}(\hlambda) := -\frac{\sqrt 2}{3} \hlambda + \alpha \left( 1 - \frac{1}{\sqrt{\htau \hlambda + 1}} \right) + \frac{\beta}{D} \left( 1 - \frac{1}{\sqrt{\htheta \hlambda + 1}} \right) = 0 \, .
\end{align}
Furthermore, there are $\kappa_j^\eps=\kappa^0_j+O(\eps)$, $j=1,2$, as in \eqref{eq:kappas} and depending on the parameters $ \alpha, \beta, D, \htau, \htheta$, such that the following holds. For $0< \eps \ll 1$ the zero eigenvalue has multiplicity two if and only if  
\begin{align} \label{eq:double_zero}
\left\{
\begin{array}{rcl}
 \kappa_1^\eps(\alpha, \beta, D, \htau, \htheta)=0 \, ,\\
 \kappa_2^\eps(\alpha, \beta, D, \htau, \htheta)\neq 0 \, ,
\end{array}
\right.
\end{align}
while it has multiplicity three if and only if
\begin{align} \label{eq:triple_zero}
\left\{
\begin{array}{rcl}
 \kappa_1^\eps(\alpha, \beta, D, \htau, \htheta)=0 \, ,\\
 \kappa_2^\eps(\alpha, \beta, D, \htau, \htheta)= 0 \, .
\end{array}
\right.
\end{align}
\end{proposition}
Hence, only small eigenvalues $ \lambda = \eps^2 \hat{\lambda} $ can lead to instabilities, so the relevant eigenvalue problem is scaled as
\begin{align}\label{eq:evp}
  \eps^2 \hlambda \, \tilde\Phi  = \calL \, \tilde\Phi \, , 
\end{align}
with $ \mathcal{L} $ given by \eqref{eq:L}. As alluded to, without directly solving the eigenvalue problem, it is a priori clear that $ \lambda = 0 $ is an eigenvalue with eigenfunction given by $ \partial_x \Zsf $. Furthermore, by varying parameters one can increase the algebraic multiplicity for the zero eigenvalue to two. In this case, we will have a corresponding Jordan block of length two since the generalized eigenfunction $\Psi$ can be readily found from the smooth family of traveling front solutions $Z^{\rm tf}$ parameterized by the speed $c$: The existence problem $-\eps^2 c \dot{ Z^{\rm tf} } = M^{-1}F(Z^{\rm tf})$ (where differentiation is meant with respect to the traveling wave coordinate $ (x-\eps^2 c) $) implies upon differentiation and evaluation at $c=0$ that we have 
\begin{align}
\nonumber
- \eps^2 \Phi = M^{-1}\partial_Z F(\Zsf)\partial_c Z^{\rm tf}|_{c=0} + b  = \calL \Psi\,,
\end{align}
since $b=\frac{\rm d}{{\rm d} c} M^{-1}F(\Zsf)|_{c=0}=0$ at the double root, which coincides with the bifurcation point of steady states. The smoothness in $\eps$ at $\eps=0$ follows from the smoothness of $\eps^2\Phi$, so that the leading order form of the first generalized eigenfunction can also be found by performing a formal asymptotic expansion and matching, see Appendix~\ref{app:formal_computation} and, in particular, Lemma~\ref{lemma:first_generalized_eigenfunction}.

By further adjustment of the parameters the algebraic multiplicity of the zero eigenvalue can be increased to three. Formal expansions (again as performed in Appendix~\ref{app:formal_computation}) suggest the existence of a second generalized eigenfunction, since the corresponding solvability condition coincides with the triple zero eigenvalue condition. We give a rigorous proof of the occurrence of a Jordan block of length three similar to the SLEP-method, i.e., the `Singular Limit Eigenvalue Problem' as developed and used in \cite{NF87} and \cite{NMIF90}. It is quite possible that the existence of the second generalized eigenfunction can also be derived from the Evans function construction used to determine the algebraic multiplicity, though we do not pursue this here.


\begin{proposition}[Jordan block structure for the zero eigenvalue]\label{proposition:jordan_block}
Let $ \eps > 0 $ be chosen sufficiently small and $ \alpha,\beta, D, \htau, \htheta $ fulfill \eqref{eq:triple_zero} such that the zero eigenvalue of the operator
\begin{align}
\nonumber
  \calL:= M(\htau, \htheta)^{-1} \partial_ZF(\Zsf;\alpha,\beta, D) 
\end{align}
is algebraically triple. Then $\calL$ possesses a Jordan chain of length 3. Specifically,  let $ \calL^* $ be the $L^2$-adjoint operator of $ \calL $ with respect to the duality product
\begin{align}
\nonumber
 \langle Z, \widetilde{Z}\rangle = \langle Z^u, \widetilde{Z}^u\rangle_{L^2} +  \langle Z^v, \widetilde{Z}^v\rangle_{L^2} +  \langle Z^w, \widetilde{Z}^w\rangle_{L^2} \, .
\end{align}
Then there are even functions $ \Phi, \Psi, \widetilde{\Psi}, \Phi^*, \Psi^*, \widetilde{\Psi}^*$ with
\begin{align} \nonumber
 \begin{array}{lcllcllcl}
 \  \calL \Phi     &=& 0 \, , &  \   \calL \Psi  &=& \eps^2 \Phi \  \, , & \ \calL \widetilde{\Psi}     &=& \eps^2 \Psi \ \, ,\\[.1cm]
  \calL^* \Phi^* &=& 0 \, , & \calL^* \Psi^* &=& \eps^2 \Phi^* \, , & \calL^* \widetilde{\Psi}^* &=& \eps^2 \Psi^* \, . 
 \end{array}
\end{align}
In particular,
\begin{equation} \nonumber
\begin{aligned}
 \langle \Phi, \Phi^* \rangle = \langle \Phi, \Psi^* \rangle = \langle \Phi^*, \Psi \rangle = 0 \,,
\end{aligned}
\end{equation}
and for any fixed $p_1,p_2,p_3\neq 0$ the (generalized) eigenfunctions $ \Phi, \Psi, \widetilde{\Psi}, \Phi^*, \Psi^*, \widetilde{\Psi}^*$ are uniquely determined by
\begin{equation}\label{e:ortho_p123}
\begin{aligned}
 p_1 := \langle \Phi, \widetilde{\Psi}^* \rangle  , \; 
 p_2 := \langle \Psi, \widetilde{\Psi}^* \rangle , \;    
 p_3 := \langle \widetilde{\Psi}, \widetilde{\Psi}^* \rangle.
\end{aligned}
\end{equation}
Moreover, the parameters and (generalized) eigenfunctions lie in a continuous family with respect to $0\leq \eps\ll 1$.
\end{proposition}

Note that the Jordan chain relations imply $p_1 = \langle \Psi, \Psi^* \rangle = \langle \widetilde{\Psi}, \Phi^* \rangle$ and $p_2=\langle \widetilde{\Psi}, \Psi^* \rangle$.

\medskip
The following subsection forms the proof in several steps, which in fact reproves Proposition~\ref{proposition:evans_function} with the SLEP-approach except for a non-degeneracy condition.


\subsection{Existence of a second generalized eigenfunction (Proof of Proposition~\ref{proposition:jordan_block})}
Recall that the existence of an eigenfunction and a first generalized eigenfunction is already settled for $\kappa_1^\eps=0$, so there exist $ \Phi, \Psi $ with $ \calL \Phi = 0, \calL \Psi = \eps^2 \Phi $ (with leading order expressions given in Appendix~\ref{app:formal_computation}). Hence, all we need to demonstrate is the existence of a second generalized eigenfunction $ \widetilde{\Psi} $ with 
\begin{align}\label{eq:existence_problem_second}
 \calL \widetilde{\Psi}  = \eps^2 \Psi \, .
\end{align}
Remark that $p_1,p_2,p_3\neq0$ in \eqref{e:ortho_p123} are the normalization constants of the generalized eigenfunctions, which we keep unspecified for now.

Upon introducing the notation
\begin{align} \nonumber
  \widetilde{\Psi}_{v, w} =  \begin{pmatrix}  \widetilde{\Psi}_{v}\\ \widetilde{\Psi}_{w} \end{pmatrix} \, , \qquad   \Psi_{v, w} =  \begin{pmatrix} \htau \Psi_{v}\\ \htheta \Psi_{w} \end{pmatrix} \, ,
\end{align}
equation \eqref{eq:existence_problem_second} can be cast in terms of a block matrix operator as
\begin{equation}\label{e:block}
\begin{pmatrix} L_\eps & \eps \A\\ \B & \C\end{pmatrix}\begin{pmatrix} \widetilde{\Psi}_u \\ \widetilde{\Psi}_{v, w} \end{pmatrix}=
\begin{pmatrix} \eps^2 \Psi_u \\ \Psi_{v,w}\end{pmatrix} \, ,
\end{equation}
with the differential and multiplication operators
\begin{align} \nonumber
 L_\eps &:= \eps^2 \partial_x^2 + 1 - 3 u^{\rm h}(x)^2 \, , \quad \C:=\mathrm{diag}(\partial_x^2-1,D^2\partial_x^2-1) \, , \quad \\ \nonumber \A &:= \begin{pmatrix} - \alpha & -\beta \end{pmatrix} \, , \quad  \B := \begin{pmatrix} 1 \\ 1 \end{pmatrix} \, .
\end{align}
We have $ L_\eps: H^2 \subset L^2 \longrightarrow L^2, S: H^2 \times H^2 \subset L^2  \times L^2 \longrightarrow L^2  \times L^2 $, and $ \A: L^2 \times L^2 \longrightarrow L^2, \B :  L^2 \longrightarrow L^2 \times L^2 $, where we suppressed the spatial domain $\R$ in each case.


\begin{lemma}[Spectrum of the operator $ L_\eps $]
\label{L4}
The operator $L_\eps: H^2 \subset L^2 \longrightarrow L^2$ is self-adjoint with maximal eigenvalue 
\begin{align}
\nonumber
\mu_\eps= \eps^2\widetilde \mu_\eps = \calO(\eps^2) \, , {\rm where} \quad  \lim_{\eps \to 0} \widetilde \mu_\eps = \widetilde\mu_0=\frac{3\sqrt{2}}{2}\left(\alpha + \frac{\beta}{D}\right) \,,
\end{align}
and with corresponding eigenfunction $\phi=\phi_\eps = \eps^{-1}\phi_0(x/\eps) + \calO(\eps)$, $\phi_0 = (\sqrt{2}/2) \mathrm{sech}^2(\cdot/\sqrt{2})$.
\end{lemma}
\begin{proof}
See Appendix~\ref{app:evp_perturbation}. $\hfill{\Box}$
\end{proof}
Consequently, we have the orthogonal splitting $L^2 = \mathrm{span}(\phi) \oplus X$, $X = \mathrm{range}(L_\eps)$ so that $L_\eps^{-1}:X\to X$ is bounded. The splitting is associated with the projections $P=\langle \cdot,\phi\rangle \phi$, i.e., $\mathrm{ker}(P)=X$, $\mathrm{range}(P)=\mathrm{span}(\phi)$ and the complementary projection $Q = \mathrm{Id} - P$. Hence, the `partial' resolvent
\begin{align}\label{eq:T_epsilon}
T_\eps:=L_\eps^{-1}Q:L^2\to L^2 
\end{align}
is bounded for each $\eps>0$. Furthermore, we have that $\C^{-1}:L^2 \times L^2 \to L^2 \times L^2$ is bounded and independent of $\eps$.

Let us now represent $ \widetilde{\Psi}_u \in L^2 $ according to the splitting induced by $ L_\eps $, that is,
\begin{align}\label{e:tPsiu}
 \widetilde{\Psi}_u = d \phi + Q[\widetilde{\Psi}_u] \, . 
\end{align}
Hence, the construction of $ \widetilde{\Psi}_u $ amounts to finding $ d $ and $ Q[\widetilde{\Psi}_u] $. Then \eqref{e:block} becomes
\begin{align}\label{eq:row_1}
 L_\eps \widetilde{\Psi}_u + \eps \A \widetilde{\Psi}_{v, w} = \eps^2 \Psi_u \quad \Longrightarrow \quad   d \mu_\eps\phi +   L_\eps Q[\widetilde{\Psi}_u] + \eps \A \widetilde{\Psi}_{v, w} = \eps^2 \Psi_u\,,
\end{align}
and 
\begin{align}\label{eq:row_2}
 B \widetilde{\Psi}_u + S \widetilde{\Psi}_{v, w} = \Psi_{v, w} \quad \Longrightarrow \quad   d B \phi + B Q[\widetilde{\Psi}_u] + S \widetilde{\Psi}_{v, w} = \Psi_{v, w}  \, .
\end{align}
Upon letting $ P $ and $ Q $ act on \eqref{eq:row_1} we get
\begin{align}
\label{eq:row_1_P}  d  \mu_\eps \| \phi \|^2  &= \langle \eps^2 \Psi_u- \eps \A \widetilde{\Psi}_{v, w}, \phi \rangle \, , \quad \textnormal{and}\\
\label{eq:row_1_Q}  L_\eps Q[\widetilde{\Psi}_u]  &=  Q[\eps^2 \Psi_u - \eps \A \widetilde{\Psi}_{v, w}]\, .
\end{align}
Using the definition of $ T_\eps $ from \eqref{eq:T_epsilon}, equation \eqref{eq:row_1_Q} can be rearranged to
\begin{align}\label{e:QtPsi}
  Q[\widetilde{\Psi}_u]  = T_\eps [\eps^2 \Psi_u - \eps \A \widetilde{\Psi}_{v, w}] \,. 
  \end{align}
Inserting back into \eqref{eq:row_2} gives
\begin{align} \nonumber
 d B \phi + B T_\eps [\eps^2 \Psi_u - \eps \A \widetilde{\Psi}_{v, w}]  + S \widetilde{\Psi}_{v, w} = \Psi_{v, w} \,.
\end{align}
After rearranging, this can be written as the equation for $  \widetilde{\Psi}_{v, w} $
\begin{align*}
 \mathcal{N}_\eps \widetilde{\Psi}_{v, w} := ( S- \eps B T_\eps \A )\widetilde{\Psi}_{v, w} = \Psi_{v, w} -  \eps^2 B T_\eps \Psi_u  - d B \phi   \, .
\end{align*}
Since $ \eps $ is small, $ \mathcal{N}_\eps $ is invertible, see Lemma~\ref{l:slepcont} below, and we get
\begin{align}\label{e:tPsivw}
 \widetilde{\Psi}_{v, w} = \mathcal{N}_\eps^{-1}[ \Psi_{v, w} -  \eps^2 B T_\eps \Psi_u  - d B \phi ]  \, .
\end{align}
Finally, we insert this expression into \eqref{eq:row_1_P} to get, as solvability condition for the existence problem \eqref{eq:existence_problem_second} of the second generalized eigenfunction $ \widetilde{\Psi} $, the equation
\begin{align} \nonumber
0 &= \left\langle \eps^2 \Psi_u- \eps \A \mathcal{N}_\eps^{-1}\left[  \begin{pmatrix} \htau {\Psi}_{v}\\ \htheta {\Psi}_{w} \end{pmatrix}  - \eps^2B T_\eps \Psi_u  - d B \phi \right], \phi \right\rangle -  \eps^2 d\widetilde\mu_\eps \| \phi \|^2 \nonumber \\
\label{eq:slep0}
  &= \left\langle  \eps^2\Psi_u  -\eps \A \mathcal{N}_\eps^{-1}\left[  \begin{pmatrix} \htau {\Psi}_{v}\\ \htheta {\Psi}_{w} \end{pmatrix}  - \eps^2 B T_\eps \Psi_u \right], \phi \right\rangle + d E  \, ,
\end{align}
with $E= \langle \eps \A \mathcal{N}_\eps^{-1} B \phi, \phi \rangle - \eps^2 \widetilde\mu_\eps \| \phi \|^2$.
If $E  \neq  0$, one could always -- for any parameter settings -- satisfy this solvability condition by choosing $ d $ accordingly and thus obtain a second generalized eigenfunction. However, from the Evans function we know that a multiple zero eigenvalue requires adjustment of parameters according to \eqref{eq:triple_zero}. Hence, it follows that $E=0$ so that $d$ remains unspecified, which is natural since \eqref{eq:existence_problem_second} does not determine $\widetilde{\Psi}$ uniquely; any multiple of $\Phi$ can be added to create another solution. 

Upon dividing  \eqref{eq:slep0} by $\eps$ we thus obtain the so-called SLEP-equation
\begin{align}\label{eq:slep}
0 = \langle  \eps\Psi_u  - \A \mathcal{N}_\eps^{-1}\left[  \begin{pmatrix} \htau {\Psi}_{v}\\ \htheta {\Psi}_{w} \end{pmatrix}  - \eps^2 B T_\eps \Psi_u \right], \phi \rangle  \, .
\end{align}
The existence of $\widetilde{\Psi}$ is now equivalent to solving \eqref{eq:slep}, and $\widetilde{\Psi}$ is then given by \eqref{e:tPsiu}, \eqref{e:QtPsi} and \eqref{e:tPsivw} for an arbitrary scalar $d$.
In order to characterize the solvability of \eqref{eq:slep}, we conclude as follows. We first compute the leading order in $\eps$ and continuity as $\eps\to 0$, which verifies that it coincides with the triple zero condition \eqref{eq:triple_zero}, and then use the implicit function theorem.


\begin{lemma}\label{l:slepcont}
It holds true for $ \eps \rightarrow 0 $  that $T_\eps:L^2\to L^2$ is uniformly bounded, with
$$\left(T_\eps + \frac12\right)\to 0 \, ,$$
from $X\cap H^2$ to $L^2$ and 
$$\calN_\eps^{-1}\to \C ^{-1}:L^2 \times L^2 \to H^2 \times H^2 \, .$$
\end{lemma}


\begin{proof}
First note that $T_\eps=L_\eps^{-1}Q$ is bounded on $L^2$ uniformly in $\eps$ since rescaling via $y=x/\eps$, which does not change the operator norm ($\|u(\cdot)\|_2 = \sqrt{\eps}\|u(\eps \cdot)\|_2$), gives $L_\eps$ as $\partial_{yy} + a_\eps(\eps y)$ with $a_\eps(x) = 1-3u^{\rm h}(x)^2$ so $a_0(x)\to -2$ as $x\to\pm\infty$ ($u^{\rm h}$ being the $u$-component of the stationary front). Hence, the rescaled $T_\eps$ is bounded uniformly in $\eps$. This implies the same for $\calN_\eps^{-1}$ so that $\calN_\eps^{-1}\to \calN_0^{-1}$ as an operator on $L^2$. Formally, $T_\eps\to T_0=a_0^{-1}\equiv -\frac 1 2$, which is however incorrect in the sense of operator converge on $L^2$ and this makes the proof a bit involved. Due to the uniform boundedness of $T_\eps$ on $L^2$ we have $T_\eps-a_0^{-1}:L^2\to L^2$ uniformly bounded in $\eps>0$, so that convergence of $(T_\eps-a_\eps^{-1})v$ with $v\in L^2$ follows from consideration of a dense subset such as $v\in H^2$.  We compute, using $T_\eps^{-1} =L_\eps$ on $X=\mathrm{range}(T_\eps)$ and that $a_0^{-1}$ commutes,
\begin{align} \nonumber
\begin{aligned} 
T_\eps-a_0^{-1} &=  a_0^{-1}  a_0 T_\eps -a_0^{-1}T_\eps T_\eps^{-1}\\
&= a_0^{-1}T_\eps (a_0 - (\eps^2\partial_{xx}+ a_\eps))\\
&=a_0^{-1}T_\eps (a_0 - a_\eps)-\eps^2 a_0^{-1}T_\eps \partial_{xx}. 
\end{aligned}
\end{align}
Since $a_\eps\to a_0$ in $L^2$ it follows that $(T_\eps-a_0^{-1})\to 0$ from $X\cap H^2$ to $L^2$ as required.

Furthermore, we have $\calN_\eps\to \calN_0=\C :H^2 \times H^2 \to L^2 \times L^2$ and for the resolvent we even have $\calN_\eps^{-1}\to \C ^{-1}:L^2 \times L^2 \to H^2 \times H^2$ since $\calN_\eps^{-1} - \C ^{-1} = \eps \C ^{-1} BT_\eps \A  \calN_\eps^{-1}$ and $ \C ^{-1} BT_\eps \A  \calN_\eps^{-1}:L^2 \times L^2\to L^2 \times L^2$ is uniformly bounded, so the claim follows from $\C ^{-1}:L^2\times L^2\to H^2\times H^2$ being bounded and constant in $\eps$. 
$\hfill{\Box}$
\end{proof}
Equipped with this, we can compute the leading order of \eqref{eq:slep} based on the following observations: Since $ \lim_{\eps \rightarrow 0} \calN_{\eps}^{-1} = S^{-1}$ and $\phi_\eps$ forms a Dirac sequence with limiting mass $\int_{\R}\phi_0 = 2 $ we have $\langle f, \phi_\eps\rangle \to 2 f(0)$ for, e.g., bounded, integrable and uniformly continuous $f$. Regarding $S^{-1}$ we use $(D^2 \partial_{xx}-1)^{-1}f = G_{D^2}*f$ with Green's function $G_{D^2}(y) = -\frac{1}{2D} \, \exp(-|y|/D)$ for localized solutions.
This immediately gives
\begin{equation}\begin{aligned}
\lim_{\eps\to 0} \langle A \calN_\eps^{-1} f,\phi\rangle &= 2(A S^{-1}[f])(0)\\&= -2 \alpha (G_1 * f_1)(0) - 2 \beta (G_{D^2} * f_2)(0) \, . \label{eq:leaging_order_term}
\end{aligned}\end{equation}
As to the leading order terms in \eqref{eq:slep}, using the leading order computation of $ \Psi $ from Lemma~\ref{lemma:first_generalized_eigenfunction} of Appendix~\ref{app:formal_computation}, we see that $\Psi_u$ is bounded so that 
\begin{align} \nonumber
\lim_{\eps\to 0} \langle \eps\Psi_u,\phi\rangle=0 \, , 
\end{align}
and also
\begin{align} \nonumber
\lim_{\eps\to 0} \langle \eps^2 \A \mathcal{N}_\eps^{-1}\left[ B T_\eps \Psi_u  \right], \phi \rangle=0 \, .
\end{align}
Hence, for the leading order analysis of the SLEP-equation \eqref{eq:slep} there is just one remaining term
\begin{align}
\nonumber
 \langle \A \mathcal{N}_\eps^{-1}\left[  \begin{pmatrix} \htau {\Psi}_{v}\\ \htheta {\Psi}_{w} \end{pmatrix} \right], \phi \rangle \, .
\end{align}

Finally, taking the explicit form  of the leading order approximation of the $ v,w $-components of $ \Psi$, see Lemma~\ref{lemma:first_generalized_eigenfunction} of Appendix~\ref{app:formal_computation}, for $f_1, f_2$ in \eqref{eq:leaging_order_term}\footnote{This amounts to solving the exact same inhomogeneous ordinary differential equations (ODEs) as in \eqref{eq:inhom_ODE} of Appendix~\ref{app:formal_computation}, whose solutions evaluated at $ x = 0 $ yield exactly \eqref{eq:inhom_ODE_sol}.} gives
\begin{align}\label{eq:slep_leading_order}
\lim_{\eps\to 0}  \langle \A \mathcal{N}_\eps^{-1}\left[  \begin{pmatrix} \htau {\Psi}_{v}\\ \htheta {\Psi}_{w} \end{pmatrix} \right], \phi \rangle = -\frac34 \left(\alpha \htau^2 + \frac{\beta}{D}\htheta^2\right).
\end{align}

In order to finish the proof of Proposition~\ref{proposition:jordan_block}, we show that \eqref{eq:slep} can be solved by an implicit function theorem: 
We seek an $\eps$-dependent family of parameters for $0\leq \eps \ll 1$ such that \eqref{eq:slep} holds, which -- in view of \eqref{eq:slep_leading_order} -- at $\eps=0$ reduces to $\kappa^0_2=0$, and this can readily be solved by adjusting parameters. For simplicity, we consider deviations from $ K= \kappa^0_2$, i.e., $K(\eps) = \kappa^0_2 + \mu(\eps) $, with  $\mu(0)=0$ so $K(0)= \kappa^0_2$. Let $h(\mu, \eps)$ denote the right hand side of \eqref{eq:slep}, i.e., we want to solve $h(\mu, \eps)=0$ for each $0\leq \eps \ll 1$ in terms of $\mu$. Notably, $h$ is continuously differentiable in $\mu$ for $0\leq \eps \ll 1$ and continuous in $\eps$ in this interval (this is pointwise for the operators) so there is continuous $\widetilde{h}$ such that
\[
h(\mu,\eps) = h(0,\eps) + h_\mu(0,\eps)\mu + \widetilde h(\mu,\eps)\mu^2 =0 \, .
\]
From $ K= \alpha \htau^2 + \beta \htheta^2/D  $ we have
\[ \partial_{K} = \htau^2 \partial_{\alpha} + \left(\htheta^2/D\right) \partial_{\beta} \]
so
\begin{align*}
\begin{aligned}
h_\mu(0,0) &= \lim_{\eps\to 0}h_\mu(0,\eps) = \partial_{K}|_{K=\kappa^0_2}
\lim_{\eps\to 0} \langle \A \mathcal{N}_\eps^{-1}\left[  \begin{pmatrix} \htau {\Psi}_{v}\\ \htheta {\Psi}_{w} \end{pmatrix} \right], \phi \rangle
\\  &= \htau^2 \left( - \frac34 \htau^2 \right) + (\htheta^2/D) \left( - \frac34 \htheta^2/D \right) \, ,
\end{aligned}
\end{align*}
which is nonzero (strictly negative) since $\htau,\htheta>0$.

By continuity $h_\mu(0,\eps)\neq 0$ for $0\leq \eps \ll1$ so that $h(\mu,\eps)=0$ is equivalent to 
\[
\mu = -h_\mu(0,\eps)^{-1}(h(0,\eps) + \widetilde h(\mu,\eps)\mu^2).
\]
This can be solved by Banach's fixed point theorem with continuous parameter since the contraction constant can be chosen uniform in $0\leq \eps\ll 1$, which yields the desired solution family $\mu=\mu(\eps)$ satisfying $h(\mu(\eps),\eps)=0$. Therefore, by local uniqueness of solutions, if $\alpha,\beta,\htau,\htheta, D$ satisfy $h(\mu(\eps),\eps)=0$ for small enough $\eps=0$ we can construct a generalized eigenfunction and also continue this, together with the parameters, uniquely in terms of $\eps$ to $\eps=0$.

This completes the proof of Proposition~\ref{proposition:jordan_block}.


\begin{remark}\label{remark:first}
The above proof works exactly the same for the existence problem of the first generalized eigenfunction $ \Psi $
\begin{align}
\nonumber
 \calL \Psi  = \eps^2 \Phi \,,
\end{align}
where $ \calL \Phi = 0 $. The SLEP equation \eqref{eq:slep} then becomes
\begin{align}\label{eq:slep_first}
 0 = \langle \eps^2 \Phi_u- \eps \A \mathcal{N}_\eps^{-1}\left[  \begin{pmatrix} \htau \Phi_{v}\\ \htheta \Phi_{w} \end{pmatrix}  - \eps^2 B T_\eps \Phi_u  \right], \phi \rangle \,.
\end{align}
Since by Lemma~\ref{lemma:eigenfunctions} of Appendix~\ref{app:formal_computation} we have that $\Phi_u = \eps^{-1}\phi_0(\cdot/\eps) + \phi_1$ for a bounded exponentially localized $\phi_1$, we have 
\begin{align} \nonumber
&\lim_{\eps\to 0} \langle \eps\Phi_u,\phi\rangle = \lim_{\eps\to 0}\eps^{-1}\langle \phi_0(\cdot/\eps),\phi_0(\cdot/\eps)\rangle = \|\phi_0\|_2^2 = \frac{2\sqrt{2}}{3} \, .
\end{align}
Moreover, for $f_1 = \htau \Phi_v = \htau \exp^{-|x|}$, $f_2 = \htheta \Phi_w = \htheta \exp^{-|x|/D}$ in \eqref{eq:leaging_order_term} we compute $(G_1 * f_1)(0)= -\frac{\htau}{2}$ and $(G_{D^2} * f_2)(0) = -\frac{\htheta}{2D}$. Therefore, \eqref{eq:slep_first} becomes
\[
0= \frac{2\sqrt{2}}{3} - \alpha\htau - \beta \htheta/D = \kappa_1^0 \, ,
\]
see \eqref{eq:kappas}, and which is precisely the leading order of \eqref{eq:double_zero} as expected. 
\end{remark}


\section{Center manifold reduction using the triple zero eigenvalue as organizing center}\label{sec:cmr}
Let us again write our original system \eqref{eq:three_component_system} in the more concise form
\begin{align}\label{e:concise}
  M(\htau, \htheta) \, \partial_t Z = F(Z;\alpha,\beta, D) \, ,
\end{align}
with $ M(\htau, \htheta) $ and $ F(Z;\alpha,\beta, D)$ as in \eqref{eq:M_F}. 
For the center manifold reduction we make an ansatz adjusted to the translation invariance of our problem, 
\begin{align}\label{eq:cmr_ansatz}
  Z(x,t) = \Zsf(x-a(t)) + \widetilde{R}(x-a(t),t) \, , 
\end{align}
with $ \Zsf = (u^{\rm h}, v^{\rm h}, w^{\rm h}) $ the stationary front (whose leading order was given in \eqref{eq:front}), so that $a$ is the position of the pseudo-front solutions. 


\begin{theorem}\label{theorem:cmr}
Let $\eps>0$ be sufficiently small and let the system parameters be a perturbation of \eqref{eq:triple_zero}. Then in a `tubular' neighborhood of the spatial translates of $ \Zsf$ in $(L^2(\R))^3$ system \eqref{e:concise}  possesses an exponentially attracting three dimensional center manifold and the reduced vector field for the center manifold variables $ (a, c, \cc) $ can be cast as follows. The position $a$ satisfies $\dot{a} = \eps^2 c$, 
which identifies $\eps^2 c$ as the velocity, and it is governed by the planar ordinary differential equations (ODE)
\begin{align}\label{eq:a_c_c_tilde_ODEs}
 \left( \begin{array}{cc} \dot{c} \\  \dot{\cc} \end{array} \right)
=
 \eps^2 \left( \begin{array}{cc} 0 & 1 \\  0 & 0 \end{array} \right)
 \left( \begin{array}{c}  c \\ \cc \end{array} \right)
 +
 \left( \begin{array}{c} 0 \\ \eps^2 G(c,\cc) \end{array} \right)
=
 \left( \begin{array}{c} \eps^2 \cc \\ \eps^2 G(c,\cc) \end{array} \right) \,,
\end{align}
where $G$ is smooth in its arguments and the parameters, and possesses the symmetry $(c,\tilde c)\to -(c,\tilde c)$.
In particular, $G(c,0)=0$ if and only if $T_\Gamma(c)=0$ to any order, where $T_\Gamma$ is the Taylor expansion of $\Gamma$ at $c=0$.
\end{theorem}
In the following, we first prove this reduction and then analyse the reduced dynamics. 


\subsection{Center manifold reduction (Proof of Theorem \ref{theorem:cmr})} 
Setting $ \eta = x-a(t) $ and substitution of \eqref{eq:cmr_ansatz} into \eqref{e:concise} gives (suppressing parameters)
\begin{align} \nonumber
 -\dot{a}(t) M \left( \Phi(\eta) + \partial_{\eta}  \widetilde{R}(\eta,t) \right) + M\partial_t \widetilde{R}(\eta,t) =  F(\Zsf(\eta) + \widetilde{R}(\eta,t)), 
\end{align}
which can be written as
\begin{align} \nonumber
 -\dot{a}(t) \left( \Phi(\eta)  + \partial_{\eta}  \widetilde{R}(\eta,t) \right)  + \partial_t \widetilde{R} =  \calL \widetilde{R} + \mathcal{N}( \widetilde{R}) \, ,
\end{align}
with $\calL, \calN$ as follows. By assumption we have $\htau=\htau_0+\check{\tau}$, $\htheta=\htheta_0+\check{\theta}$ (and likewise for $\alpha, \beta, D$), where subindex zero denotes the parameters values of a parameter set at which \eqref{eq:triple_zero} holds. Now $\calL:= M(\htau_0, \htheta_0)^{-1} \partial_ZF(\Zsf)$, which, after some simplifications, yields
\begin{align} \nonumber
\begin{aligned}
  \mathcal{N}(\widetilde{R}) &= \left(\begin{array}{ccc} 0 & 0 & 0 \\ 0 & -\frac{\check{\tau}}{\htau_0+\check{\tau}} & 0 \\ 0 & 0 & -\frac{\check{\theta}}{\htheta_0+\check{\theta}}  \end{array} \right) M(\htau_0,\htheta_0)^{-1}  \partial_ZF(\Zsf)\widetilde{R} \\&  \qquad \qquad- \left(\begin{array}{c} 3 (\Zsf)^u (\widetilde{R}^u)^2 + (\widetilde{R}^u)^3  \\ 0 \\ 0 \end{array} \right) \,.
\end{aligned}
\end{align}
Using the Jordan block structure of $ \mathcal{L} $ (as stated in Proposition~\ref{proposition:jordan_block}) we refine the ansatz to
\[
\widetilde{R}(x-a(t),t) = b(t) \Psi(x- a(t)) + \widetilde{b}(t) \widetilde{\Psi}(x- a(t)) + R(x-a(t),t),
\]
where $R$ is $L^2$-orthogonal to the adjoint generalized kernel spanned by $\Phi^*, \Psi^*, \widetilde\Psi^*$. Hence,
\begin{align}
\nonumber
  Z(x,t) = \Zsf(x-a(t)) + b(t) \Psi(x- a(t)) + \widetilde{b}(t) \widetilde{\Psi}(x- a(t)) + R(x-a(t),t) \, , 
\end{align}
and, suppressing the dependence of $\Psi, \widetilde{\Psi}$ and $R$ on $\eta$, 
\begin{align}  \label{eq:before_projection}
\begin{aligned}
&-\dot{a}(t) \left( \Phi + b(t) \Psi' + \widetilde{b}(t) \widetilde{\Psi}' + \partial_{\eta} R(t) \right) + \dot{b}(t)\Psi + \dot{\widetilde{b}}(t)\widetilde{\Psi} + \partial_t R(t) = \\
& \quad \quad \quad \quad  \eps^2 b(t) \Phi + \eps^2 \widetilde{b}(t) \Psi +  \calL R(t) + \mathcal{N}(b(t) \Psi + \widetilde{b}(t)\widetilde{\Psi} + R(t)) \, . 
\end{aligned}
\end{align}
Recall that $ \Zsf $ is odd, so $ \Phi = (\Zsf)' $ is even, and $ \Psi, \widetilde{\Psi}, \Phi^*, \Psi^*, \widetilde{\Psi}^*$ are all also even, such that  their derivatives are odd, and, hence,
\begin{align} \nonumber
 \langle \Psi', \Phi^* \rangle = \langle \widetilde{\Psi}', \Phi^* \rangle = \langle \Psi', \Psi^* \rangle = \langle \widetilde{\Psi}', \Psi^* \rangle = \langle \Psi', \widetilde{\Psi}^* \rangle = \langle \widetilde{\Psi}', \widetilde{\Psi}^* \rangle =  0 \, .    
\end{align}
Executing the projections on \eqref{eq:before_projection} (and again suppressing parameter dependence) then gives the equations on the generalized kernel as
\begin{align} \nonumber
   \begin{pmatrix}d_1 \\  d_2 \\ d_3 \end{pmatrix}\dot{a}
 +
A
 \begin{pmatrix}\dot{a} \\ \dot{b} \\ \dot{\widetilde{b}} \end{pmatrix} 
=
 \begin{pmatrix}0 &  p_1 & p_2\\ 0 & 0 & p_1 \\ 0 & 0 & 0 \end{pmatrix}
\begin{pmatrix} a \\ \eps^2 b \\ \eps^2 \widetilde{b} \end{pmatrix}
&+ \nonumber
\begin{pmatrix} \langle \mathcal{N}[ b \Psi + \widetilde{b}\widetilde{\Psi} + R], \widetilde{\Psi}^* \rangle \\  \langle \mathcal{N}[ b \Psi + \widetilde{b}\widetilde{\Psi} + R], \Psi^* \rangle \\ \langle \mathcal{N}[ b \Psi + \widetilde{b}\widetilde{\Psi} + R], {\Phi}^* \rangle \end{pmatrix}\, ,
\end{align}
with 
\begin{align} \nonumber
A:= \begin{pmatrix} p_1  & p_2 & p_3 \\  0 & p_1 & p_2 \\ 0 & 0 & p_1 \end{pmatrix},\; 
d_1:= -\langle \partial_{\eta}R, \widetilde{\Psi}^* \rangle , \; d_2:= -\langle \partial_{\eta}R, \Psi^* \rangle , \;  d_3:= -\langle \partial_{\eta}R, {\Phi}^* \rangle,
\end{align}
and $ p_1, p_2, p_3  $ as in \eqref{e:ortho_p123} of Proposition~\ref{proposition:jordan_block}. Note that $d_1,d_2,d_3\to 0$ as $R\to 0$ by integration by parts; in particular $|d_1|<1$ in the range we consider due to the ansatz \eqref{eq:cmr_ansatz} from the tubular vicinity of $\Zsf$. 

Multiplying the last equation by 
\begin{align} \nonumber
A^{-1} =  \left( \begin{array}{ccc}1/p_1 & -p_2/p_1^2 & (p_2^2 - p_1 p_3)/p_1^3 \\ 0 & 1/p_1 &  -p_2/p_1^2 \\ 0 & 0 &  1/p_1  \end{array} \right)
\end{align}
gives the first form of the reduced system
\begin{align} %
\label{eq:after_projection_and_normal_form} 
  \dot{a} \left( \begin{array}{c} q_1 \\  q_2 \\ q_3 \end{array} \right)
 +
 \left( \begin{array}{c} \dot{a} \\ \dot{b} \\ \dot{\widetilde{b}} \end{array} \right)
 = 
 \left( \begin{array}{ccc} 0 &  1 & 0\\ 0 & 0 & 1 \\ 0 & 0 & 0 \end{array} \right)
 \left( \begin{array}{c} a \\ \eps^2 b \\ \eps^2 \widetilde{b} \end{array} \right) 
 +
 \left( \begin{array}{c} N_1[b,\widetilde{b},R ] \\ N_2[b,\widetilde{b},R ] \\ N_3[b,\widetilde{b},R ] \end{array} \right) \, ,
\end{align}
where
\begin{align}
\nonumber 
 q_1:= \frac{d_1}{p_1} - \frac{d_2 p_2}{p_1^2} + d_3\frac{p_2^2 - p_1 p_3}{p_1^3} \, , \qquad q_2:= \frac{d_2}{p_1} - \frac{d_3 p_2}{p_1^2} \, , \qquad  q_3:= \frac{d_3}{p_1} \, ,
\end{align}
and
\begin{align} \nonumber 
 \left( \begin{array}{c} N_1[b,\widetilde{b},R ] \\ N_2[b,\widetilde{b},R ] \\ N_3[b,\widetilde{b},R ] \end{array} \right)
 := A^{-1} \left( \begin{array}{c} \langle \mathcal{N}[b\Psi + \widetilde{b}\widetilde{\Psi} + R ], \widetilde{\Psi}^* \rangle \\  \langle \mathcal{N}[b\Psi + \widetilde{b}\widetilde{\Psi} + R ], \Psi^* \rangle \\ \langle \mathcal{N}[b\Psi + \widetilde{b}\widetilde{\Psi} + R ], {\Phi}^* \rangle \end{array} \right) \, .
\end{align}
Observe now that the right-hand-side of these equations does not depend explicitly on $ a $. In particular, using the rescaling $(\b, \widetilde{\b}) = (b/(q_1 + 1), \tilde{b}/(q_1 + 1)) $ and denoting
\begin{align} \nonumber %
 \underline{N}_j[\b,\widetilde{\b},R ] :=  N_j[(q_1 + 1)) \b,(q_1 + 1)) \widetilde{\b},R ]\\
g(\b,\widetilde{\b},R) :=  \eps^2 \b +  \underline{N}_1[\b,\widetilde{\b},R ]/(q_1 + 1) \, ,
\end{align}
we can rewrite the system \eqref{eq:after_projection_and_normal_form} as
\begin{align} \nonumber 
\begin{array}{rcl}
   \dot{a} \quad &=&  \eps^2 \b +  \underline{N}_1[\b,\widetilde{\b},R ]/(q_1 + 1) \, , \\[.2cm]
     \left(\begin{array}{c} \dot{\b} \\ \dot{\tilde{\b}} \end{array} \right)
    &=& 
   \left(\begin{array}{cc} 0 & 1 \\ 0 & 0 \end{array} \right) \left(\begin{array}{c} \varepsilon^2 \b \\ \varepsilon^2 \tilde{\b} \end{array} \right)
   +
    \left(\begin{array}{c} 
     \underline{N}_2[\b,\widetilde{\b},R ] - q_2 g(\b,\widetilde{\b},R) \\
     \underline{N}_3[\b,\widetilde{\b},R ] - q_3 g(\b,\widetilde{\b},R) 
    \end{array} \right) \, .
\end{array}
\end{align}
The spectral properties noted in Proposition~\ref{proposition:jordan_block}, the semi-linear problem structure and smoothness of the nonlinearity  imply the existence of an exponentially attracting center manifold for $0<\eps\ll 1$ (see, e.g. \cite{HI11} Thm. 3.22). This means $
 R = H(\b, \widetilde{\b}),
$
with smooth function $H$ independent of $a$ since the right hand side is independent of $a$ (cf.\ \cite{HI11}, Thm 3.19). Hence, we get the reduced system  
\begin{align}\label{eq:a_b_b_tilde_ODEs_}
   \left\{
\begin{array}{rcl}
   \dot{a} \quad &=&  \varepsilon^2 \b + \varepsilon^2 F_1[\b, \tilde{\b}]  \, , \\[.2cm]
   \dot{\b} \quad &=&  \varepsilon^2 \tilde{\b} + \varepsilon^2 F_2[\b, \tilde{\b}]  \, , \\[.2cm]
   \dot{\tilde{\b}} \quad &=& \varepsilon^{2} F_3[\b, \tilde{\b}]  \, , \\[.2cm]   
\end{array}
     \right.
\end{align}
with
$ F_1[\b, \tilde{\b}] = \varepsilon^{-2} \underline{N}_1[\b,\widetilde{\b},H(\b, \widetilde{\b}) ]/(1+q_1), $
$ F_2[\b, \tilde{\b}] = \varepsilon^{-2}( \underline{N}_2[\b,\widetilde{\b},H(\b, \widetilde{\b}) ] - q_2 g(\b,\widetilde{\b},H(\b, \widetilde{\b})) )$, and  
$ F_3[\b, \tilde{\b}] = \varepsilon^{-2} ( \underline{N}_3[\b,\widetilde{\b},H(\b, \widetilde{\b}) ] - q_3 g(\b,\widetilde{\b},H(\b, \widetilde{\b})) )$, and
where the seemingly singular scaling of $ F_j $ will be justified in the next section. We will not explicitly compute $ F_j $ in terms of projections and the expansion of the center manifold, but rather perform another transformation that connects the coordinates on the center manifold with the velocity.

\begin{lemma}
There exists a near-identity change of variables $ (\b, \tilde{\b}) \mapsto (c, \tilde{c}) $ such that \eqref{eq:a_b_b_tilde_ODEs_} becomes
\begin{align} \nonumber 
\begin{array}{rcl}
   \dot{a} \quad &=&  \varepsilon^2 c  \, , \\[.2cm]
   \left(\begin{array}{c} \dot{c} \\ \dot{\tilde{c}} \end{array} \right)
    &=& 	
   \varepsilon^2  \left(\begin{array}{cc} 0 & 1 \\ 0 & 0 \end{array} \right)    \left(\begin{array}{c} c \\  \tilde{c} \end{array} \right)
   +
    \left(\begin{array}{c} 0 \\ \varepsilon^2 G(c, \tilde{c})  \end{array} \right) \, .
\end{array}
\end{align}
\end{lemma}

\begin{proof}
We first make the near-identity change of variables $ c := \b+ F_1[\b, \tilde{\b}] $, which can be inverted locally to $ \b = \mathcal{A}_1(c, \tilde{\b}) $. So $ \dot{a} = \eps^2 c$, and taking a derivative gives 
\begin{align} \nonumber 
 \dot{c} = \eps^2 \tilde{\b} + \eps^2   \underline{F}_2[c, \tilde{\b}]
\end{align}
with
\begin{align} \nonumber 
\underline{F}_2[c, \tilde{\b}]
 =& F_2[\mathcal{A}_1(c, \tilde{\b}), \tilde{\b}]+  \partial_1 F_1[\mathcal{A}_1(c, \tilde{\b}), \tilde{\b}](\tilde{\b}  + F_2[\mathcal{A}_1(c, \tilde{\b}), \tilde{\b}]) \\ \nonumber&  + \partial_2 F_1[\mathcal{A}_1(c, \tilde{\b}), \tilde{\b}]F_3[\mathcal{A}_1(c, \tilde{\b}), \tilde{\b}]  \, .
\end{align}
A further near-identity change of variables by $ \tilde{c} :=  \tilde{\b} +  F_2[c, \tilde{\b}] $ (again locally invertible to $ \tilde{\b} = \mathcal{A}_2(c, \tilde{c}) $) gives $\dot{c} = \eps^2 \tilde{c}$. Finally, taking a derivative as before we obtain
\begin{align} \nonumber 
 \dot{\tilde{c}} = \eps^2  G(c, \tilde{c}) \, ,
\end{align}
where $ G $ can be specified in terms of $ \underline{F}_2, F_1, F_2, F_3 $ analogous to the previous step, though we make no direct use of this. 
$\hfill{\Box}$
\end{proof}
The advantage of this reformulation is that equilibria in these coordinates, $G(c,0)=0$, are traveling fronts with this $c$-value as its velocity to any expansion order. Regarding symmetry, the reflection symmetries of \eqref{e:concise} $x\to -x$ and $Z\to-Z$ imply that $H$ can be chosen to respect this in the reduced coordinates, which gives the claimed symmetry with respect to the reflection 
$$(\eta,a,c,\tilde c)\to -(\eta,a,c,\tilde c) \, .$$ 
This completes the proof of Theorem~\ref{theorem:cmr}.


\subsection{Dynamics of the reduced system on the center manifold} 
Since the planar ODE for the velocity $ c $ has the form \eqref{eq:a_c_c_tilde_ODEs} one can already anticipate a Bogdanov-Takens type bifurcation scenario. The type of unfolding is determined by the degeneracies in the expansion of $G(c,\cc)$ in \eqref{eq:a_c_c_tilde_ODEs}. That is,
\begin{equation}\label{e:cmfexpand}
G(c,\cc)|_{\eps=0} = g_1 c + g_3 c^3 + \sigma_1 c^5 + \cc(g_2 + g_4 c^2 + \sigma_2 c^4) + O(\cc^2) + h.o.t. \, ,
\end{equation}
where $g_j$ are functions of the system parameters. 
We will later select system parameters to unfold the bifurcation, and for now denote by $\mu$ an abstract selection of system parameters, so that $g_j=g_j(\mu)$ and $\mu=0$ is the bifurcation point.

\begin{definition}\label{DEF} Concerning the possible degeneracies we say that we have a 
\begin{itemize}
\item\textit{symmetric Bogdanov-Takens (SBT)} iff $g_1|_{\mu=0} = g_2|_{\mu=0} = 0$ and $g_3 g_4|_{\mu=0}\neq 0$, see Figure~\ref{f:unfoldsketch11};
\item\textit{symmetric Bodganov-Takens with butterfly imprint (SBTB)} iff $g_1|_{\mu=0}=g_2|_{\mu=0}=g_3|_{\mu=0}=0$ and $g_4|_{\mu=0}\neq 0$; and
\item\textit{symmetric Bodganov-Takens with degeneracy (SBTD)} iff $g_1|_{\mu=0}=g_2|_{\mu=0}=g_4|_{\mu=0}=0$ and $g_3|_{\mu=0}\neq 0$.
\end{itemize}
\end{definition}
A normal form for these cases, as well as the additional option of $g_1|_{\mu=0}=g_2|_{\mu=0}=g_3|_{\mu=0}=g_4|_{\mu=0}$, has been derived in \cite{Knobloch}. This normal form is, in the slow time scale $'=\eps^{-2} d/dt$, given by
\begin{align} \nonumber
\left\{
\begin{array}{rcl}
{\underline{c}}'  &=& \tilde{\underline{c}}\\
{\tilde{\underline{c}}}' &=& \underline{g}_1 \underline{c} + \underline{g}_3 \underline{c}^3 + \underline{\sigma}_1 \underline{c}^5 + \underline{\cc} \left( \underline{g}_2 + \underline{g}_4 \underline{c}^2 + \underline{\sigma}_2  \underline{c}^4 \right) \,,
\end{array}
\right.
\end{align}
where the underscores emphasize that, in general, an additional coordinate change is required to reach this normal form. 
However, in the SBT case this is not needed and we can also ignore $\underline{\sigma}_1$ and $\underline{\sigma}_2$. 
In contrast, in the SBTB case the additional coordinate change depends on the coefficient $\partial_c\partial_{\cc}^2G(0,0)$ \cite{Knobloch}. Since our approach does not provide access to compute $\partial_c\partial_{\cc}^2G(0,0)$, it does not allow to rigorously unfold this case. 

\begin{figure}[!ht]
\centering
\includegraphics[scale=0.9]{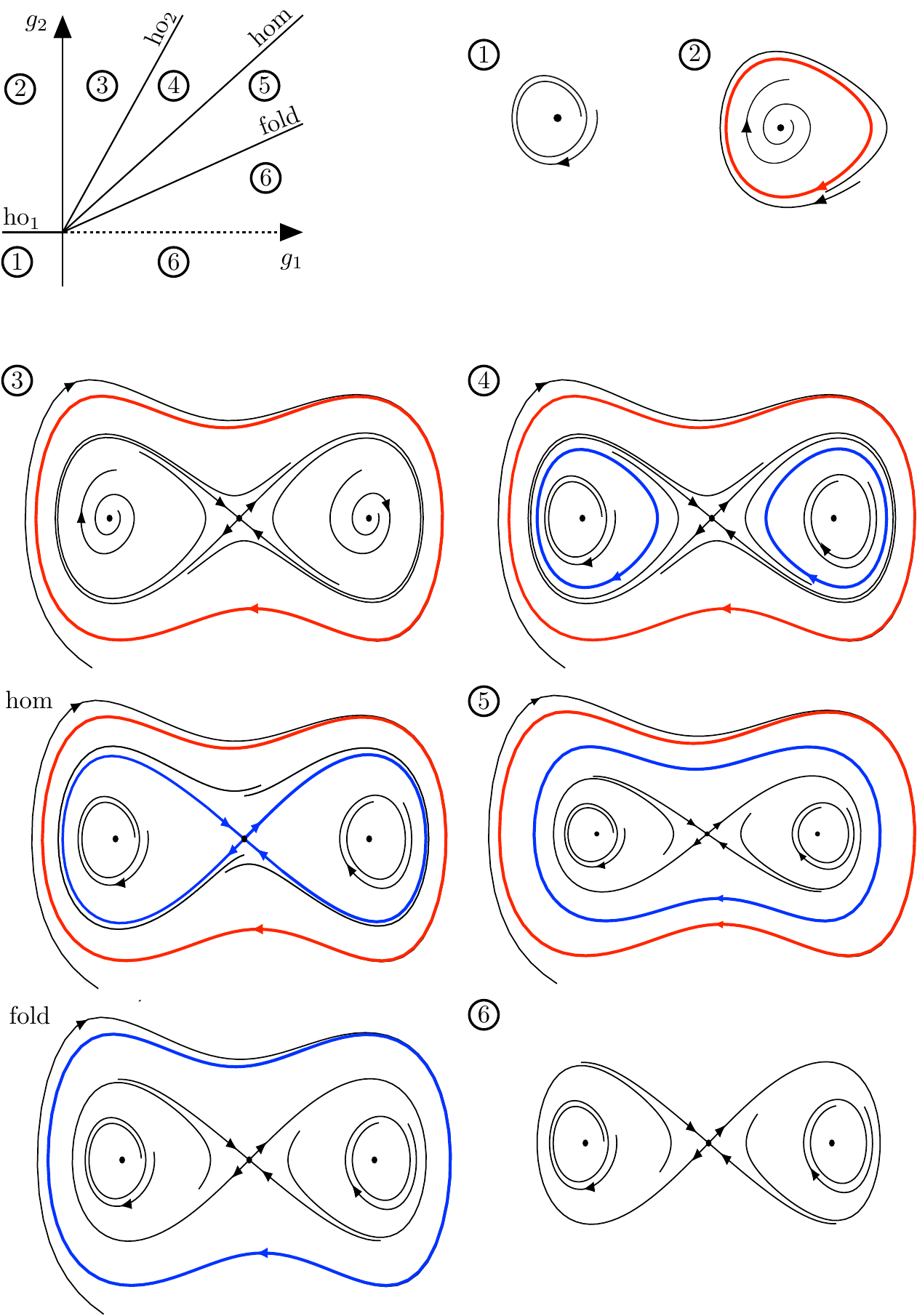}\\
\caption{Bifurcation diagram, and sketches of the associated phase planes, of a symmetric Bogdanov-Takens point in the case $g_3,g_4<0$. The periodic orbits plotted in red are stable, those in blue unstable.
See also Figures 2 and 4 from \cite{Carr}.}
\label{f:unfoldsketch11}
\end{figure}

In this paper, we focus on the SBT case, so that two parameters $\mu=(\mu_1,\mu_2)$ suffice, and follow the analysis in \cite{Carr} to check the relevant terms in the right-hand-side $ G $ in \eqref{eq:a_c_c_tilde_ODEs} directly, by explicitly computing some of its derivatives. As alluded to, in these computations 
we exploit the analytic information on the (leading order) existence condition and critical eigenvalues for uniformly traveling fronts, i.e., for the fixed points \eqref{eq:a_c_c_tilde_ODEs}. 
Since all coefficients are continuous at $\eps=0$ it suffices to focus on the leading order at $\eps=0$. For convenience, we first summarize the information needed from the existence and stability analysis, see also \cite{CDHR15}. 

The leading order in $ \eps $ of the Evans function arising from the stability analysis of uniformly traveling fronts has been determined in \cite{CDHR15} to be
\begin{align} \nonumber 
 {\calD}(\hlambda,c) = -\frac{\sqrt 2}{6} \hlambda + \alpha \widetilde{\calD}(\hlambda, c \htau, \hat{\tau})+ \frac{\beta}{D} \widetilde{\calD}\left(\hlambda, c\frac{\hat{\theta}}{D}, \hat{\theta} \right) = 0 \, ,
\end{align}
with
\begin{align} \nonumber %
  \widetilde{\calD}(\hlambda, \rho_1, \rho_2) = \left( \frac{1}{\sqrt{\rho_1^2 + 4}} -  \frac{1}{\sqrt{\rho_1^2 + 4(\hat{\lambda}\rho_2 + 1)}} \right) \, .
\end{align}


\begin{lemma}\label{lemma:taylor}
Let $ \kappa_1^0, \kappa_2^0, \kappa_3^0 $ be as in \eqref{eq:kappas}. Then the Taylor expansion in $ c = 0 $ of the leading order existence condition from \eqref{eq:existence_condition} for uniformly traveling fronts with velocity $ \eps^2 c $ for $ \gamma = 0 $ is given by
\begin{align}\label{eq:T_gamma}
 T_{\Gamma}(c) = \frac12 \kappa_1^0 c - \frac{1}{16} \kappa_3^0 c^3 + k c^5 + O(c^7) \, ,
\end{align}
where $k= \frac{3}{256} \left( \alpha \htau^5 + \beta \frac{\htheta^5}{D^5} \right)$ does not vanish if $\kappa_1^0=\kappa_3^0=0$. Furthermore, the leading order of the Evans function\footnote{Note that the Evans function and $\calE$ are meaningful only for choices of $c$ and system parameters such that a traveling front with velocity $c$ exists.} arising from the stability analysis of uniformly traveling fronts (with translational eigenvalue factored out)  has the form
\[
\calE(\hlambda,c) =  \frac{\mathcal{D}(\hlambda, c)}{\hlambda} = a_0(c)+ a_1(c)\hlambda  + a_2(c)\hlambda^2 + a_3(c)\hlambda^3 + O(\hlambda^4)
\] and  is an even function of $c$ with expansions of the coefficients given by $ a_i = \sum_{j \geq 0} a_{2j,i} \ c^{2j} $ with
\begin{align}
\left\{
\begin{array}{rclcrclcrcl}
a_{00} &=& \dfrac14 \kappa_1^0\, ,  \quad
a_{20}  =- \dfrac{3}{32}\kappa_3^0 \, , \quad
a_{01} = - \dfrac{3}{16} \kappa_2^0\, , \quad
a_{21} =   \dfrac{15}{128}\left(\alpha \htau^4 + \frac{\beta}{D^3} \htheta^4 \right) \, , \\[0.3cm]
a_{02} &=&   \dfrac{5}{32} \left( \alpha \htau^3 + \frac{\beta}{D} \htheta^3 \right) \, , \quad 
a_{03} =  -\dfrac{35}{256}  \left( \alpha \htau^4 + \frac{\beta}{D} \htheta^4 \right) \, .
\end{array}
\right. \label{eq:a_ij}
\end{align}
\end{lemma}
In the following, we will make use of the Taylor coefficients $ a_{ij} $ of the Evans function to derive expressions for the coefficients of the reduced system on the center manifold and discuss the unfolding of its bifurcation structure. Hence,
  \begin{align}\label{eq:a_ij_mu}
   a_{ij} = a_{ij}(\mu) \, ,
  \end{align}
where $ \mu = (\mu_1, \mu_2) $ is some choice of unfolding parameters. 


\begin{proof}[of Lemma~\ref{lemma:taylor}]
A straightforward computation yields the Taylor expansion of the existence condition \eqref{eq:existence_condition} (with $\gamma=0$). In order to verify that $ k \neq 0$ at $ \kappa_1^0 = \kappa_3^0 = 0 $, we can use that $ \kappa_1^0 = \kappa_3^0 = 0 $ can be expressed as 
$$
\alpha = - \frac{2 \sqrt{2}}{3} \frac{\htheta^2}{\htau(D^2 \htau^2-\htheta^2 )}\, , \qquad
\beta = \frac{2 \sqrt{2}}{3} \frac{D^3 \htau^2}{\htheta(D^2 \htau^2-\htheta^2)}\,,$$
with automatically nonzero denominator at $ \kappa_1^0 = \kappa_3^0 = 0 $.
This gives
$$  k|_{\kappa_1^0 = \kappa_3^0 = 0} = -\frac{ \sqrt{2}}{128} \htau^2 \left(\frac{\htheta}{D} \right)^2  \neq 0 \, . $$
Another straightforward calculation gives
\begin{align} \nonumber 
\begin{aligned}
a_0(c) &= -\frac{\sqrt 2}{6} + 2\alpha  \htau \left({\cal{F}}(c \htau)\right)^3 + 2 \frac{\beta}{D}\htheta \left({\cal{F}}(c \htheta/D)\right)^3
\, ,\\
a_1(c) &= - 6 \alpha \htau^2 \left({\cal{F}}(c \htau)\right)^5 -   6 \frac{\beta}{D}\htheta^2 \left({\cal{F}}(c \htheta/D)\right)^5 \,, \\
a_2(c) &=  20 \alpha \htau^3 \left({\cal{F}}(c \htau)\right)^7 +  20 \frac{\beta}{D}\htheta^3 \left({\cal{F}}(c \htheta/D)\right)^7 \,,  \\
a_3(c) &= -70 \alpha \htau^4  \left({\cal{F}}(c \htau)\right)^9 - 70 \frac{\beta}{D}\htheta^4 \left({\cal{F}}(c \htheta/D)\right)^9\,, 
\end{aligned}
\end{align}
with 
$
  {\cal{F}}(\rho_1) = (\rho_1^2+4)^{-1/2}.
$
Further direct computations yield the claimed Taylor expansions.
$\hfill{\Box}$
\end{proof}
\subsubsection{Unfolding of the codimension two case}
Throughout this section, we denote by $ \mu = (\mu_1, \mu_2) $ any choice of parameters in \eqref{eq:three_component_system} with $(\kappa_1^0,\kappa_2^0)|_{\mu = 0}=0$, and denote with $\nabla_\mu$ the gradient with respect to $ \mu $. 
The following proposition allows us to identify and unfold the SBT case.


\begin{proposition}\label{prop:structure_G}
The reduced system on the center manifold \eqref{eq:a_c_c_tilde_ODEs} has, to leading order in $\eps$,  the form 
\begin{align}
  \nonumber
  \left\{
\begin{aligned}
\dot{c}  &= \eps^2 \cc\\
\dot{\tilde c} &= \eps^2 G(c,\tilde c,\mu) = \eps^2 \left(G_1(c,\mu) + \widetilde c {G}_2(c,\mu)\right) + o(\eps^2) \, ,
\end{aligned} \right.
\end{align}
where
\begin{align}
\begin{array}{rclcl}
 G_1(c,\mu)   &=& [g_{10} + g_{11}\cdot \mu]c & + & g_{30} c^3 + h.o.t. \, , \quad {\textnormal{and}}\\
 G_2(c,0,\mu) &=& \ g_{20} + g_{21}\cdot \mu & + & g_{40} c^2 +  h.o.t. \, ,
\end{array}
\label{eq:g_ij}
\end{align}
with $ g_{j0} \in \mathbb{R}, g_{j1} \in \mathbb{R}^2 $ (and, hence, linear functions $g_{j1}\cdot \mu $). Moreover, with $ a_{ij} $ from \eqref{eq:a_ij} and the notation from \eqref{eq:a_ij_mu}, 
\begin{align*}
 g_{10} = 0 \, , \,\, g_{20} = 0 \, , \,\, 
 g_{30} = -\frac{1}{3}\frac{a_{20}(0)}{a_{02}(0)} \, , \,\, 
 g_{40} = -\frac{1}{a_{02}(0)} \left( a_{21}(0) - \frac{a_{20}(0) a_{03}(0)}{a_{02}(0)} \right) \, ,
\end{align*}
that is, the linear part is, as expected, a Jordan block of length two at the organizing center, and
\begin{align*}
 g_{11} = - \frac{\nabla_{\mu} a_{00}(0) }{ a_{02}(0)} \, , \qquad
 g_{21} = - \frac{1}{a_{02}(0)} \left[ \nabla_{\mu} a_{01}(0) - \frac{ \nabla_{\mu} a_{00}(0) a_{03}(0)}{a_{02}(0)} \right] \, ,
\end{align*}
\end{proposition}


\begin{corollary}\label{cor:structure_G_parameters}
The coefficients $ g_{ij} $ from Proposition~\ref{prop:structure_G} satisfy
\begin{align}
 g_{11} &= \frac{6\sqrt2}{5\htau\htheta} \left( \nabla_\mu \kappa_1^0 \right)|_{\mu=0} \,,\label{e:g1mu}\\
 g_{21} &= -\frac{3}{5\sqrt2\htau\htheta}\left.\left( 3\nabla_\mu\kappa_2^0 - \frac{7}{2}(\htau+\htheta)\nabla_\mu \kappa_1^0
\right)\right|_{\mu=0} \,, \label{e:g2mu}
\end{align}
as well as 
\begin{align*}
 g_{30} &= \left.-\frac{3}{20\sqrt{2}}\left(\frac{\kappa_3^0}{\htau\htheta}\right)\right|_{\mu=0} = \left.\frac{1}{5}\frac{D^2\htau-\htheta}{D^2(\htau-\htheta)} \right|_{\mu=0} ,  \\
 g_{40} &=\left. -\frac{3}{40}\frac{3(D^2\htau^2-\htheta^2)+7\htau\htheta(1-D^2)}{D^2 (\htau -\htheta)} \right|_{\mu=0}  \, .
\end{align*}
In particular, $g_{30}=0$ is equivalent to $\kappa_3^0|_{\mu=0}=0$ and in this case $g_{40}=-\frac 3 4 \htau<0$. Conversely, if $g_{40}=0$ then $7\htau>3\htheta$ and $g_{30} = \frac{2\htau}{7\htau-3\htheta}> 0$.
Moreover, if $g_{30}<0$ then $g_{40}<0$. 
\end{corollary}


\begin{remark}
The fact that $g_{30}=g_{40}=0$ is not possible implies that the degeneracies of higher order than the SBT case are either the SBTB case or the SBTD case, see Definition~\ref{DEF}.
\end{remark}


\begin{proof}[of Proposition~\ref{prop:structure_G} and Corollary~\ref{cor:structure_G_parameters}]
Since eigenvalues are invariant under coordinate changes, the eigenvalues of the linearization of \eqref{eq:a_c_c_tilde_ODEs} in equilibria coincide (in the sense of Taylor expansions) with the two small eigenvalues of the operator $ \mathcal{L}$. Recall only these eigenvalues (and the fixed zero eigenvalue) of $ \mathcal{L}$ are close to the imaginary axis and satisfy $ \lambda = \eps^2 \hlambda, \hlambda = \mathcal{O}(1) $. From this, we infer that $ G = \mathcal{O}(1) $ with respect to $ \varepsilon $ and also the two small eigenvalues $ \hlambda_j^\eps$, $j=1,2$, coincide with the eigenvalues of the linearization of \eqref{eq:a_c_c_tilde_ODEs} in equilibria. All quantities are at least continuous in $\eps$ at $\eps=0$ and we discuss the leading order next. Since fixed points of \eqref{eq:a_c_c_tilde_ODEs} are roots of $ G(c,0;\mu) = 0 $, for some functions $ G_1, G_2 $ we have
\begin{align} \nonumber 
G(c,\tilde c;\mu)|_{\eps=0} = G_1(c;\mu) + \widetilde c G_2(c,\widetilde c;\mu) \,,
\end{align}
where $G_1 $ has the same zeros as $T_\Gamma$ from \eqref{eq:T_gamma}. 
Linearizing the resulting system (in slow time) and evaluating at $\widetilde c=0$ gives the matrix
\begin{align} \nonumber 
\left( \begin{array}{cc} 0 & 1 \\  \partial_c G_1(c;\mu) &  G_2(c,0;\mu) \end{array} \right) \, ,
\end{align}
whose characteristic equation
\begin{align}\label{eq:char_eq_martix}
 \hlambda^2 - G_2(c,0;\mu) \hlambda - \partial_c G_1(c;\mu) =0\, ,
\end{align}
has the same roots as the (reduced) Evans function $\calE$ (in the sense of expansion). Precisely two of these roots vanish at the SBT point $\kappa_1^0 = \kappa_2^0 = 0$ and $\calE$ is analytic. The Weierstrass preparation theorem, cf.\ e.g.\ \cite{ChowHale}, thus yields 
\begin{equation}\label{e:weier}
\calE(\hlambda,c;\mu) =\left(\hlambda^2 +\ta_1(c,\mu) \hlambda + \ta_0(c,\mu)\right)\tE(\hlambda,c;\mu),
\end{equation}
for unique $\ta_0, \ta_1$ and non-vanishing $\tE$, all being holomorphic in $c$ and system parameters in a neighborhood of the SBT point. 
Comparing \eqref{eq:char_eq_martix} and~\eqref{e:weier} implies that
\begin{align}\label{eq:taylor_comparison}
  \mathcal{T}_{c,\mu}(-G_2) =  \mathcal{T}_{c,\mu}(\ta_1) \,, \qquad   \mathcal{T}_{c,\mu}(-\partial_c G_1) =  \mathcal{T}_{c,\mu}(\ta_0) \,,
\end{align}
that is, the Taylor expansions (including $c$ and parameters) of $-\ta_1$ and $G_2(\cdot,0)$, as well as $-\ta_0$ and $\partial_c G_1(\cdot)$ coincide, respectively. Since we know the Taylor expansion of $\calE$ from Lemma~\ref{lemma:taylor} we can employ a two step procedure to derive the formulas for $ g_{ij} $ in \eqref{eq:g_ij}:
\begin{description}
 \item{\underline{Step 1}:} Compute  the unknown $ \ta_j $ recursively from
 \begin{align*}
  \calE(\hlambda,c;\mu) &=(\hlambda^2 +\ta_1(c,\mu) \hlambda + \ta_0(c,\mu))\tE(\hlambda,c;\mu) \\ &= a_0(c,\mu) + a_1(c,\mu) \lambda + a_2(c,\mu) \lambda^2 + a_3(c,\mu) \lambda^3 + \ldots \, , 
 \end{align*}
 where the $a_j$'s are given in Lemma~\ref{lemma:taylor}.\\
 \item{\underline{Step 2}:} Use \eqref{eq:taylor_comparison} to determine the expressions for the $ g_{ij} $. \\
\end{description}

Let us now turn our attention to the recursion. By analyticity, 
\[
\tE(\hlambda,c;\mu)=\te_0(c,\mu) + \te_1(c,\mu) \lambda + \te_2(c,\mu) \lambda^2 + \te_3(c,\mu) \lambda^3 + \ldots
\]
and therefore
\begin{align*}
  \calE(\hlambda,c;\mu)
  &=(\hlambda^2 +\ta_1(c,\mu) \hlambda + \ta_0(c,\mu))\tE(\hlambda,c;\mu) \\
  &= \overbrace{[\ta_0(c,\mu)\te_0(c,\mu)]}^{\overset{!}{=}a_0(c,\mu)} + \overbrace{[\ta_1(c,\mu)\te_0(c,\mu)+ \ta_0(c,\mu)\te_1(c,\mu)]}^{\overset{!}{=}a_1(c,\mu)}  \lambda\\[.01cm]
  & \hspace{1cm} + \underbrace{[\te_0(c,\mu)+ \ta_1(c,\mu)\te_1(c,\mu)+ \ta_0(c,\mu)\te_2(c,\mu)]}_{\overset{!}{=}a_2(c,\mu)}  \lambda^2 + \ldots \, . 
 \end{align*}
At $ (c, \mu) = (0,0) $ we, hence, get
\begin{align*}
  a_0(0,0) & =  \ta_0(0,0)\te_0(0,0) \, ,\\
  a_1(0,0) & =  \ta_1(0,0)\te_0(0,0)+ \ta_0(0,0)\te_1(0,0) \, ,\\
  a_2(0,0) & =  \te_0(0,0)+ \ta_1(0,0)\te_1(0,0)+ \ta_0(0,0)\te_2(0,0) \, , \ldots \, .
\end{align*}
Since we know that $ a_0(0,0) = \frac14 \kappa_1^0|_{\mu = 0} =  0 $ and $\te_0(0,0) \neq0  $, we can conclude that $ \ta_0(0,0) = 0 $, so $ a_1(0,0) =  \ta_1(0,0)\te_0(0,0) $. But since $ a_1(0,0) =  -\frac{3}{16} \kappa_2^0|_{\mu = 0} =  0 $ and $\te_0(0,0) \neq0  $, this also implies $ \ta_1(0,0) = 0 $. Hence, the above recursion simplifies to
\begin{align}\label{eq:e_j}
  a_{j+2}(0,0) & =  \te_j(0,0) \, ,  \quad j \geq 0 \, .
\end{align}
In particular,
\begin{align*}
  \te_0(0,0) = a_{02}(0) \, , \quad \mathrm{and} \quad \te_1(0,0) = a_{03}(0) \, ,
\end{align*}
with $ a_{ij} $ from \eqref{eq:a_ij}. Equipped with this information, we go to the second step and compare the leading order terms in the Taylor expansions \eqref{eq:taylor_comparison} to infer the $ g_{j0} $, that is, the coefficients at the SBT point. We immediately get
\begin{align*}
 g_{10} = - \ta_0(0,0) = 0 \, , \quad \mathrm{and} \quad g_{20} = - \ta_1(0,0) = 0 \, .
\end{align*}
In order to compare the higher order terms in the Taylor expansion, we have to take derivatives of the above recursion. This implies (where a similar reasoning as before gives $ \partial_c \ta_0(0,0) = 0 $)
\begin{align*}
\left.\partial_c^2 [\ta_0(c,\mu)\te_0(0,0)]\right|_{(c,\mu) = (0,0)} = \partial_c^2 a_0(0,0) \\
   \,\Leftrightarrow \, 
 \partial_c^2 \ta_0(c,\mu) &= \frac{\partial_c^2 a_0(0,0)}{\te_0(0,0)}  = 2 \frac{a_{20}(0)}{a_{02}(0)} .
\end{align*}
The claimed expression for $g_{30}$ now immediately follows from
\begin{align*}
 3 \cdot 2 \ g_{30} = - \partial_c^2 \ta_0(0,0) \, ,
\end{align*}
A slightly longer, but analogous reasoning gives
\begin{align*} \partial_c^2 [\ta_1(c,\mu)\te_0(c,\mu)&+ \ta_0(c,\mu)\te_1(c,\mu)] \big|_{(c,\mu) = (0,0)} = \partial_c^2 a_1(0,0) \\ 
  \quad \Leftrightarrow \quad 
 \partial_c^2 \ta_1(c,\mu) &= \frac{1}{\te_0(0,0)} \left( \partial_c^2 a_1(0,0) - (\partial_c^2 \ta_0(0,0))\te_1(0,0) \right) \\ &=  \frac{1}{\te_0(0,0)} \left( \partial_c^2 a_1(0,0) -\frac{\partial_c^2 a_0(0,0)}{\te_0(0,0)} \te_1(0,0) \right)\, .
\end{align*}
So we infer the claimed $ g_{40} $ from
\begin{align*}
 2 \ g_{40} = - \partial_c^2 \ta_1(0,0) = - \frac{1}{a_{02}(0)} \left( 2a_{21}(0) - 2 \frac{a_{20}(0)}{a_{02}(0)}a_{03}(0) \right).
\end{align*}

It remains to compute the expressions that yield the unfolding parameters. To this end, we need the linear terms in the Taylor expansion of $ \ta_0, \ta_1 $ with respect to $ \mu $, that is, we again have to differentiate, but this time with respect to $ \mu $.  Using again $ \ta_0(0,0) = 0 $ this gives
\begin{align*}
 \nabla_{\mu} a_0(0,0) =  \nabla_{\mu}[ \ta_0(c,\mu)\te_0(c,\mu)]\big|_{(c,\mu) = (0.0)} 
 \quad \Leftrightarrow \quad  
 \nabla_{\mu} \ta_0(0,0) = \frac{\nabla_{\mu} a_0(0,0)}{\te_0(0,0)} \, ,
\end{align*}
so that
\begin{align*}
 g_{11} =  - \nabla_{\mu} \ta_0(0,0) = - \frac{\nabla_{\mu} a_{00}(0)}{a_{02}(0)} \, ,
\end{align*}
and
\begin{align*}
 \nabla_{\mu} a_1(0,0) =  \nabla_{\mu}&[ \ta_1(c,\mu)\te_0(c,\mu)+ \ta_0(c,\mu)\te_1(c,\mu) ]\big|_{(c,\mu) = (0.0)} \\
 \quad  \Leftrightarrow \quad 
 \nabla_{\mu} \ta_1(0,0) 
 &= \frac{1}{\te_0(0,0)} \left(\nabla_{\mu} a_1(0,0)- \nabla_{\mu} \ta_0(0,0) \te_1(0,0) \right) \\
& = \frac{1}{\te_0(0,0)} \left(\nabla_{\mu} a_1(0,0)- \frac{\nabla_{\mu} a_0(0,0)}{\te_0(0,0)} \te_1(0,0) \right) \, . 
\end{align*}
Now using \eqref{eq:e_j}, we get
\begin{align*}
 g_{21} =  - \frac{1}{a_{02}(0)} \left(\nabla_{\mu} a_{21}(0)- \frac{\nabla_{\mu} a_{00}(0)}{a_{02}(0)} a_{03}(0) \right) \, .
\end{align*}
Finally, the statement about $ g_{31}, g_{41} $ follows by continuity. This completes the proof of the lemma.

In order to verify the expressions in the corollary, we use the explicit expressions from \eqref{eq:a_ij}. Noting that $ \kappa_1^0 = \kappa_2^0 = 0$ implies
\begin{equation}\label{e:albe0}
\alpha =\frac{2 \sqrt{2} \htheta}{3 \htau (\htheta-\htau)}, \, 
\beta= \frac{2 \sqrt{2} D \htau}{3 (\htau - \htheta) \htheta}\,,
\end{equation}
we thus find
\[
\te_0(0,0) =  a_{02}\big|_{\kappa_1^0=0, \kappa_2^0=0}  =  -\frac{5\sqrt2}{48} \htau\htheta <0.
\]
Furthermore,
\begin{align*}
g_{11} &= \left. \dfrac{48}{5\sqrt2\htau\htheta}\nabla_\mu a_0\right|_{\mu=0}
= \left. \frac{12}{5\sqrt2\htau\htheta}\nabla_\mu \kappa_1^0\right|_{\mu=0}\\
g_{21} &= \left.-\frac{48}{80\sqrt2\htau\htheta}\left(
3\nabla_\mu\kappa_2^0 - \frac{7}{2}(\htau+\htheta)\nabla_\mu \kappa_1^0
\right)\right|_{\mu=0},
\end{align*} 
where we used that $a_3|_{c,\mu=0} = \left. -\frac{35}{256}\left(\alpha \htau^4+\frac{\beta}{D}\htheta^4\right)\right|_{\mu=0}$ gives 
\[
\left. \frac{a_3}{a_2}\right|_{c,\mu=0} = -\left. \frac{7}{8}(\htau+\htheta)\right|_{\mu=0}.
\]
Finally,
\begin{align*}
g_{30} &= \left. -\frac{3 \sqrt2}{20}\left(\frac{\kappa_3^0}{\htau\htheta}\right)\right|_{\mu=0},\\
g_{40} &=\left. -\frac{3}{40}\frac{3(D^2\htau^2-\htheta^2)+7\htau\htheta(1-D^2)}{D^2 (\htau -\htheta)}\right|_{\mu=0},
\end{align*}
where the latter follows from simplifying, at $\kappa_1^0=\kappa_2^0=0$, the expression
\begin{align*}
&- \frac{1}{a_{02}(0)}\left(a_{21}(0)-\frac{a_{03}(0)}{a_{02}(0)}a_{20}(0)\right)
\\
&\qquad =\left. \frac{24\sqrt2}{5\htau\htheta}\left(
\frac{15}{128}\left(\alpha \htau^4+\frac{\beta}{D^3}\htheta^4\right)+
\frac{35}{256}\left(\alpha \htau^4+\frac{\beta}{D}\htheta^4\right)\frac{24\sqrt2}{5\htau\htheta}\frac{3}{32}\kappa_3^0
\right)\right|_{\mu=0}\,,
\end{align*}
using that at $\mu=0$ we have 
\[
\alpha \htau^4+\frac{\beta}{D^3}\htheta^4 =  -\dfrac{2\sqrt{2}}{3}\htau\htheta\frac{D^2\htau^2-\htheta^2}{D^2(\htau-\htheta)}\,,\;
\alpha \htau^4+\frac{\beta}{D}\htheta^4= - \dfrac{2\sqrt{2}}{3}\left(\htau \htheta (\htau + \htheta)\right).
\]
The special case $g_{30}=0$ means $D^2=\frac{\htau}{\htheta}$ and $\htheta\neq \htau$, which imply $g_{40}=-\frac 3 4 \htau<0$ based on these formulas. Moreover, if $\htau=\frac 3 7 \htheta$ then $g_{40}=-\frac{9}{28}\htheta\neq 0$, and $g_{40}=0$ is equivalent to $\htau\neq \htheta$ and $D^2=\frac\htheta\htau\frac{3\htheta-7\htau}{3\htau-7\htheta}$, which must be positive. The numerator is positive if $3 \htheta> 7 \htau$ and the denominator is positive if $7\htheta<3 \htau $. So, both must be negative, which holds for $\frac 3 7 < \frac\htheta\htau < \frac 7 3$. In this case $g_{30} = \frac{2\htau}{7\htau-3\htheta}>0$. 

Concluding the proof, we show that $g_{30}<0$ implies $g_{40}<0$. For brevity, define $y:=\frac\htheta\htau$. Then, $g_{30}<0$ is equivalent to $1 < y < D^2$ due to the standing assumption that $D>1$. So, we assume that $1 < y < D^2$ and note that
$g_{40}>0$ is equivalent to $h(y):=\frac 7 3(D^2-1) - \frac{D^2}{y}+y<0$. 
Since $h'(y) = \frac{D^2}{y^2} + 1>0$, the function $h$ is strictly monotonically increasing and $h(1) = \frac43(D^2-1)>0$. 
In other words, $h(y)>0$ for $1 < y < D^2$ and, thus,   $g_{40}<0$ if $g_{30}<0$.
$\hfill{\Box}$
\end{proof}

In the statement and proof of Proposition~\ref{prop:structure_G}, we used \eqref{e:albe0} so that $g_3, g_4$ at $\mu=0$ are independent of $\alpha,\beta$. Then it is natural to fix $\htau, \htheta$ and choose $\mu$ affine in $\alpha, \beta$ as 
\begin{align}\label{e:mu}
\mu=(\alpha, \beta) - \frac{2 \sqrt{2} }{3(\htheta-\htau)} \left(\frac{\htheta}{\htau},  -\frac{D \htau}{\htheta}\right).
\end{align}
With this choice \eqref{e:g1mu}, \eqref{e:g2mu} together with \eqref{eq:kappas} give the matrix
\begin{equation}\label{e:albemat}
\left.\partial_{(\alpha,\beta)}\begin{pmatrix}g_1\\g_2\end{pmatrix} \right|_{\mu=0}= 
\frac{3\sqrt2}{5}
\begin{pmatrix}
\frac 2 \htheta & \frac{2}{D\htau}\\
\frac{\htau+7\htheta}{4\htheta} & \frac{\htheta+7\htau}{4D\htau}
\end{pmatrix}\,.
\end{equation}
We can now prove the main result concerning the unfolding of the SBT case, i.e., the triple zero eigenvalue of the PDE without additional degeneracy. The corresponding bifurcation diagram in the case $g_{30}, g_{40}<0$ for the expansion \eqref{e:cmfexpand} is as plotted in Figure~\ref{f:unfoldsketch11}.


\begin{theorem}
Let $\htau, \htheta>0$, $D>1$ and $0<\eps\ll 1$. For $D^2\htau\neq \htheta$ and $D^2\htau\,(3\htau-7\htheta)\neq \htheta\,(3\htheta-7\htau)$ the parameters $\alpha, \beta$ unfold the bifurcation point $\kappa_1^0=\kappa_2^0=0$ of fronts in the sense of unfolding the SBT case for the reduced vector field \eqref{eq:a_c_c_tilde_ODEs} within the odd symmetry class of $G$.
\end{theorem}
 \begin{proof}
Due to Proposition~\ref{prop:structure_G},  $D^2\htau\neq \htheta$ and $D^2\htau\,(3\htau-7\htheta)\neq \htheta\,(3\htheta-7\htau)$ imply $g_{30}, g_{40}\neq 0$, which are precisely the nondegeneracy conditions (H1), (H2) in \cite[Chapter 4]{Carr} for the vector field \eqref{eq:a_c_c_tilde_ODEs} with expansion \eqref{e:cmfexpand}; the condition (H3) in that reference holds since \eqref{eq:a_c_c_tilde_ODEs} is in second order form. 

The unfolding parameters in \cite[Chapter 4]{Carr} are $g_1\, (=\partial_c G_1(0))$ and $g_2\, (=G_2(0))$. The matrix \eqref{e:albemat} has determinant $\frac{54}{25}\frac{\htau-\htheta}{D\htau\htheta}\neq 0$ since $\htau\neq \htheta$ if $\kappa_1^0=\kappa_2^0=0$. It follows that near the SBT point $g_1=g_2=0$ the mapping $(g_1, g_2) \mapsto (\alpha,\beta)$ is invertible. 
This and the signs of $g_{30},g_{40}$ persists for $0<\eps\ll 1$ by continuity. $\hfill{\Box}$
\end{proof}

\begin{remark}
Recall that due to Corollary~\ref{cor:structure_G_parameters} the case $g_{30}<0$ and $g_{40}>0$ cannot occur in \eqref{e:cmfexpand}. Otherwise, this would correspond to reflecting the case $g_{30}<0$ and $g_{40}<0$ by  $(\cc, g_2, t)\to -(\cc,g_2,t)$. This means that the unfolding (with $\gamma=0$) cannot generate stable `traveling breathers', i.e., periodically oscillating pseudo-fronts with nonzero average speed. In other words, there is no Hopf bifurcation from traveling fronts with $c\neq0$ to stable periodic orbits in the reduced ODE. 
\end{remark}

\begin{remark}\label{r:gam}
Invoking the parameter $\gamma$ breaks the odd symmetry of $G$ as the existence condition for traveling fronts directly shows. Note that $\alpha, \beta$ will also unfold bifurcations for $|\gamma|>0$ sufficiently small. In particular, consider the Hopf bifurcation for $g_{30}, g_{40}$, $g_1<0$ from stationary fronts to `standing breathers' with zero average speed, labelled ho$_1$ in Figure~\ref{f:unfoldsketch11}. Changing $\gamma$ to a non-zero value will move this bifurcation point to a Hopf bifurcation that creates stable traveling breathers. We plot a numerical example in Figure~\ref{f:movper}.
\end{remark}


\subsection{Numerical continuation and simulation}\label{s:num}
In this section, we present numerical computations that illustrate and corroborate the results of the previous sections.
We use the software package {\tt pde2path} \cite{p2p} for numerical continuation and bifurcation computations as well as simulations of the full PDE \eqref{eq:three_component_system}.  Our focus lies on recovering numerically the theoretical bifurcations sketched in Figure~\ref{f:unfoldsketch11}. As a starting point we take the setting from \cite[Fig. 11]{CDHR15}, which shows a periodic solution found by direct numerical simulation near a triple root. We fix 
\begin{equation}\label{e:numfix}
\eps=0.03, \htheta=10, \htau=4.21, D= 2.2
\end{equation}
and use a domain $[-L,L]$ with homogeneous Neumann boundary conditions.  Unless noted otherwise we take $L=10$, which turns out to be large enough so that longer domains do not noticably  change the results. 

The numerical simulations of the time evolution for pseudo-fronts were done using the `freezing method'  \cite{Beyn,p2psym}, where the domain moves effectively along with the traveling front in a comoving frame $\zeta(t)=x-a(t)$ with velocity $\dfrac{d}{dt}a=c$ of the pseudo-front (recall that in the analysis the velocity was rescaled to slow time). The instantaneous velocity is determined in each time step through the orthogonality condition to the group orbit of the translation symmetry given by
\[
c(t) = \frac{\langle M^{-1} F(Z), Z_x\rangle}{\|Z_x\|^2_2}.
\]
In the comoving $\zeta$-coordinate, we can  work on a relatively short spatial interval and with a fixed grid that is refined near the center, where the gradients are concentrated. We compute the `position' based on this velocity as $a(t) = \int_0^t c(s) ds$, but note that in general dynamic pseudo-fronts move relative to the $\zeta$ variable. For instance, in bifurcating periodic solutions the zero intersection of the $u$-component is not stationary in $\zeta$ but moves periodically.
\begin{figure}[tbp]
\begin{center}
\begin{tabular}{ccc}
\includegraphics[width=0.3\textwidth]{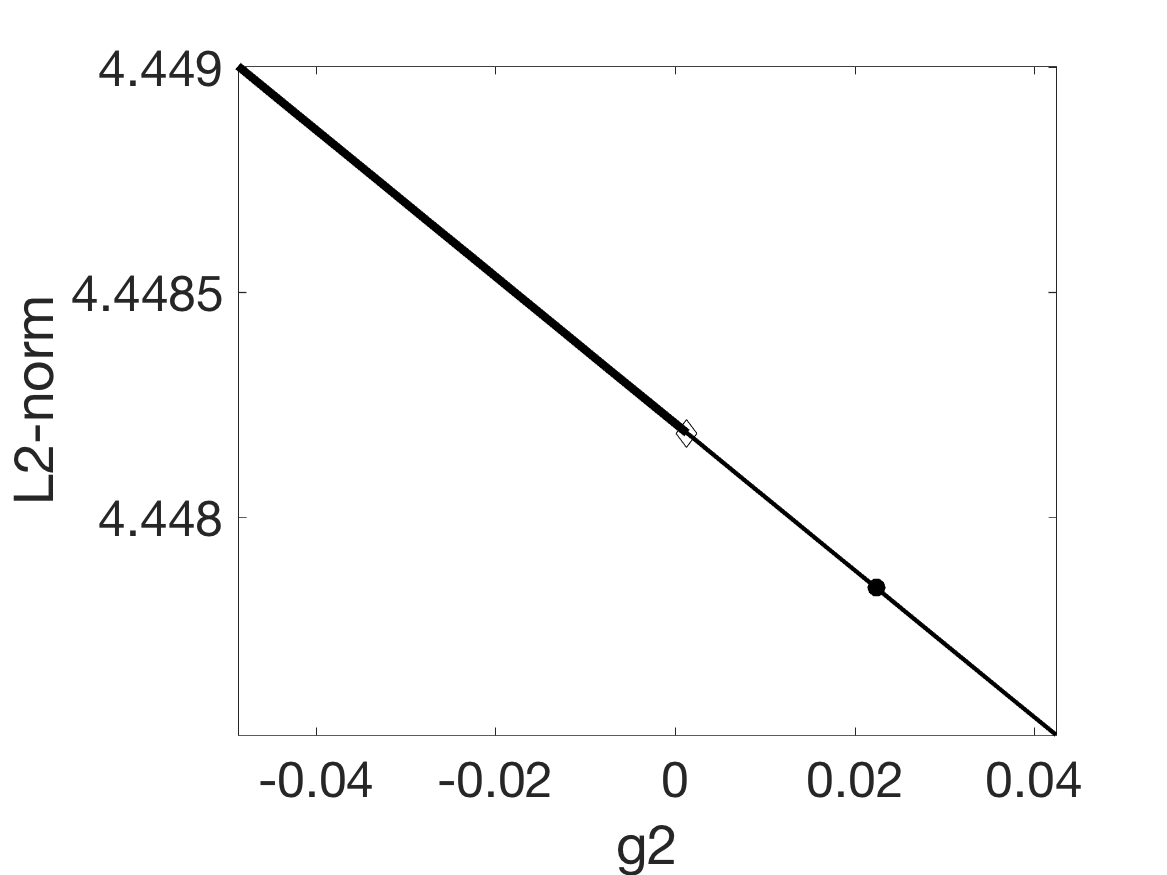}
&\includegraphics[width=0.3\textwidth]{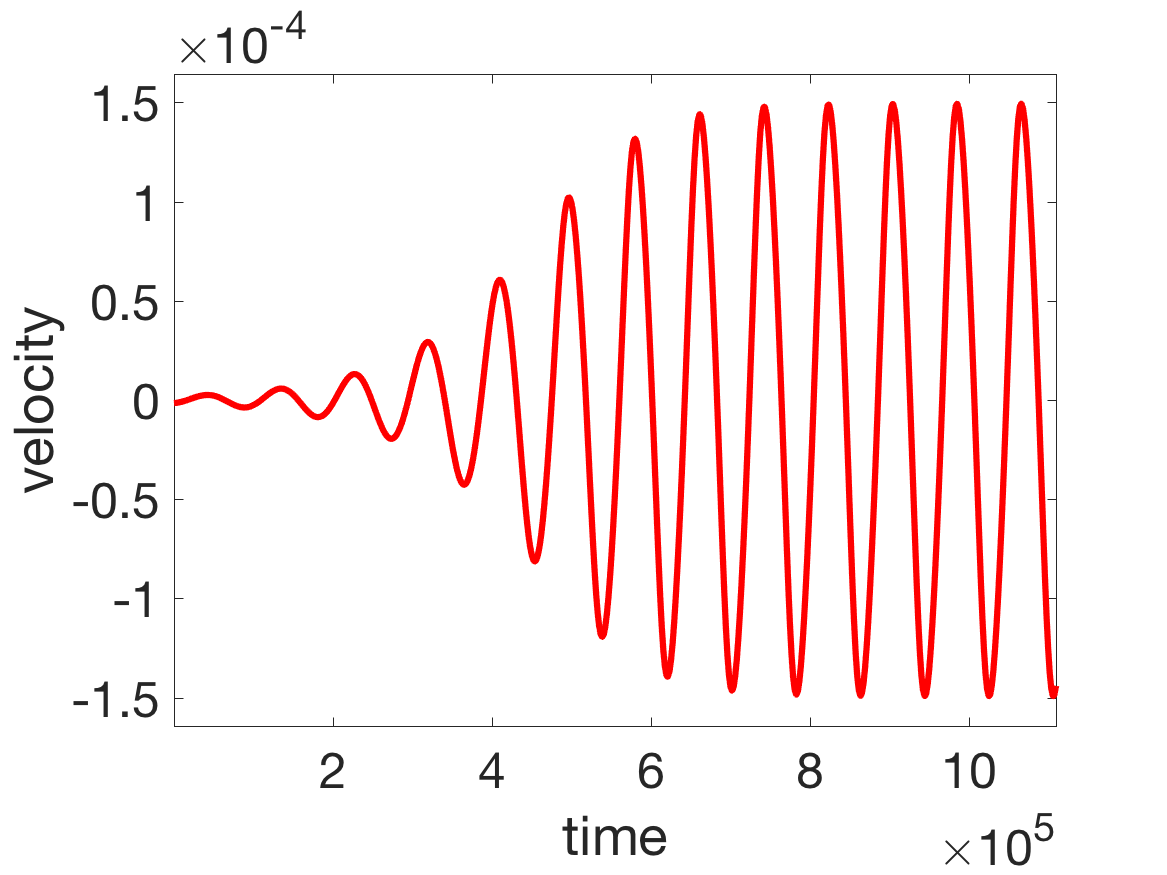}
&\includegraphics[width=0.3\textwidth]{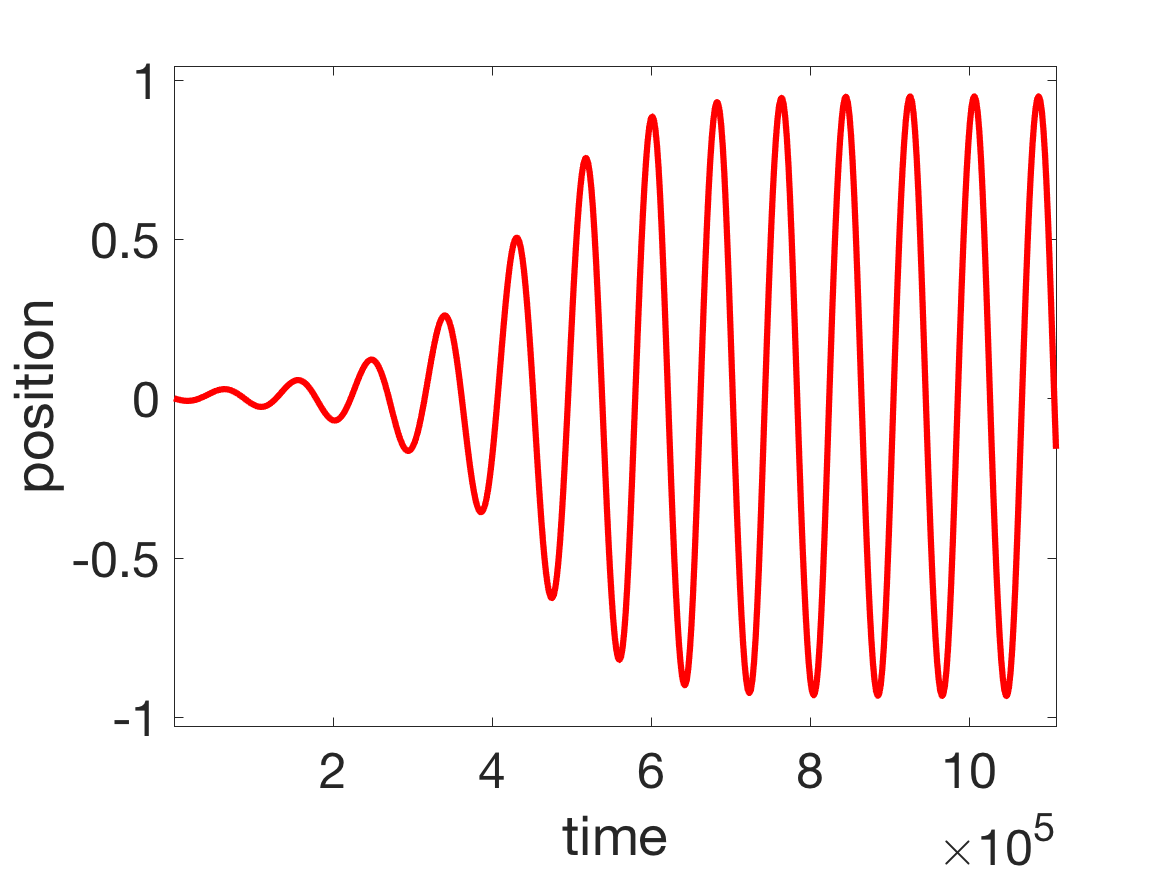}\\
(a) & (b) & (c)
\end{tabular}
\caption{From region 1 to 2 of Figure~\ref{f:unfoldsketch11}, i.e., $\g_3\approx-0.074$, $\g_4\approx  -3.14$ with $\g_1\approx-0.005$, and \eqref{e:numfix}.
%
(a) Branch of fronts from numerical continuation (stable thick, unstable thin) destabilising at a Hopf bifurcation point (diamond) at $\g_2\approx0.001$, theoretically predicted for $g_{30},g_{40}<0$ at $g_2=0$. (b,c) Plots of velocity and position from a simulation of a perturbation from the solution at the bullet in (a), where $\g_2\approx 0.02$, and PDE parameters in \eqref{eq:three_component_system} are $\alpha\approx 0.45$, $\beta\approx-0.24$ and \eqref{e:numfix}.
\label{f:triple-in-g2}}
\end{center}
\end{figure}

\begin{figure}[tbp]
\begin{center}
\begin{tabular}{ccc}
\includegraphics[width=0.3\textwidth]{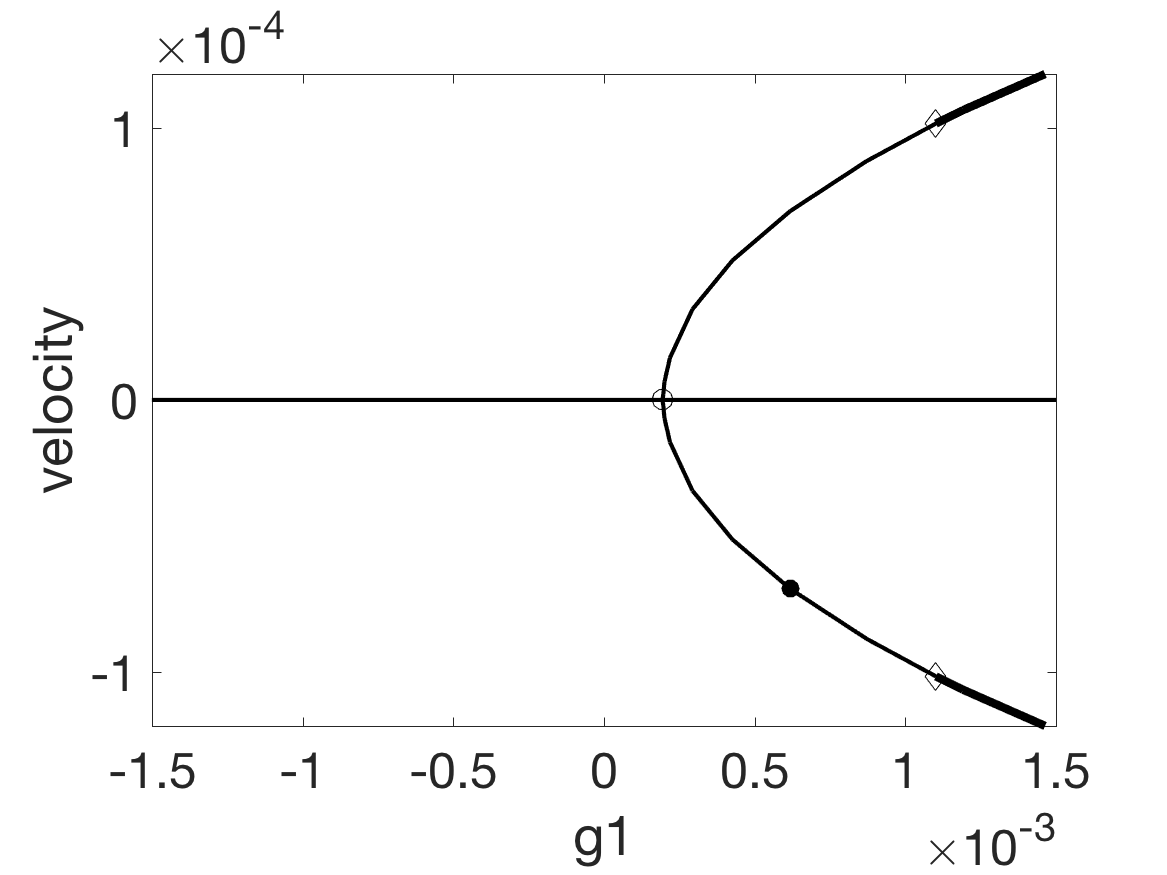}
&\includegraphics[width=0.3\textwidth]{vels-in-g1-g2pos-b-pt6.png}
&\includegraphics[width=0.3\textwidth]{pos-in-g1-g2pos-b-pt6.png}\\
(a) & (b) & (c)
\end{tabular}
\caption{From region 2 to 3,4 of Figure~\ref{f:unfoldsketch11}, i.e., $\g_3\approx-0.074$, $\g_4\approx  -3.14$ with $\g_2=0.042$, and \eqref{e:numfix}.
%
(a) Branches of fronts from numerical continuation (stable thick, unstable thin) connected at a pitchfork bifurcation at $\g_1\approx0.0002$ stabilising at Hopf points (diamonds) at $\g_1\approx0.001$, matching the theoretical prediction. 
(b,c) Plots of velocity and position from a simulation of a perturbation from the solution at the bullet in (a), where $\g_1\approx 0.0006$, and PDE parameters in \eqref{eq:three_component_system} are $\alpha\approx 0.44$, $\beta\approx-0.19$ and \eqref{e:numfix}. See also Figure~\ref{f:numintro}.
\label{f:triple-in-g1-g2pos}}
\end{center}
\end{figure}

Recall that the theoretical values of the center manifold coefficients $g_j$ from the previous sections were computed in the singular limit $\eps=0$. Since $\eps>0$ in the numerical computations, we expect the values differ slightly from the numerical ones, which we therefore denoted by $\g_j$. We approximate $\g_3$ and $\g_4$ using the formulas from Corollary~\ref{cor:structure_G_parameters} and take $\alpha, \beta$ as affine functions of $\g_1$ and $\g_2$ through \eqref{e:mu} and \eqref{e:albemat}.

The results plotted in Figures~\ref{f:triple-in-g2} and \ref{f:triple-in-g1-g2pos} correspond to the crossing from region 1 of Figure~\ref{f:unfoldsketch11} to region 2 -- a Hopf bifurcation -- and further to regions 3 and 4 -- a pitchfork bifurcation followed by another Hopf bifurcation. The crossing from region 1 to region 6 in Figure~\ref{f:unfoldsketch11}, i.e., crossing the $g_2$-axis with $g_2<0$, corresponds to the results plotted in Figure~\ref{f:triple3}. Here a pitchfork bifurcation occurs, near which the emerging heteroclinic connection lies in a one-dimensional center manifold and is thus monotone. However, the phase portrait plotted for region 6 in Figure~\ref{f:unfoldsketch11}, illustrates the case of complex leading eigenvalues of the bifurcated stable equilibrium. This highlights the underlying two-dimensional dynamics, which we find also numerically,  as plotted in Figure~\ref{f:triple3}(b).

\begin{figure}[tbp]
\begin{center}
\begin{tabular}{ccc}
\includegraphics[width=0.3\textwidth]{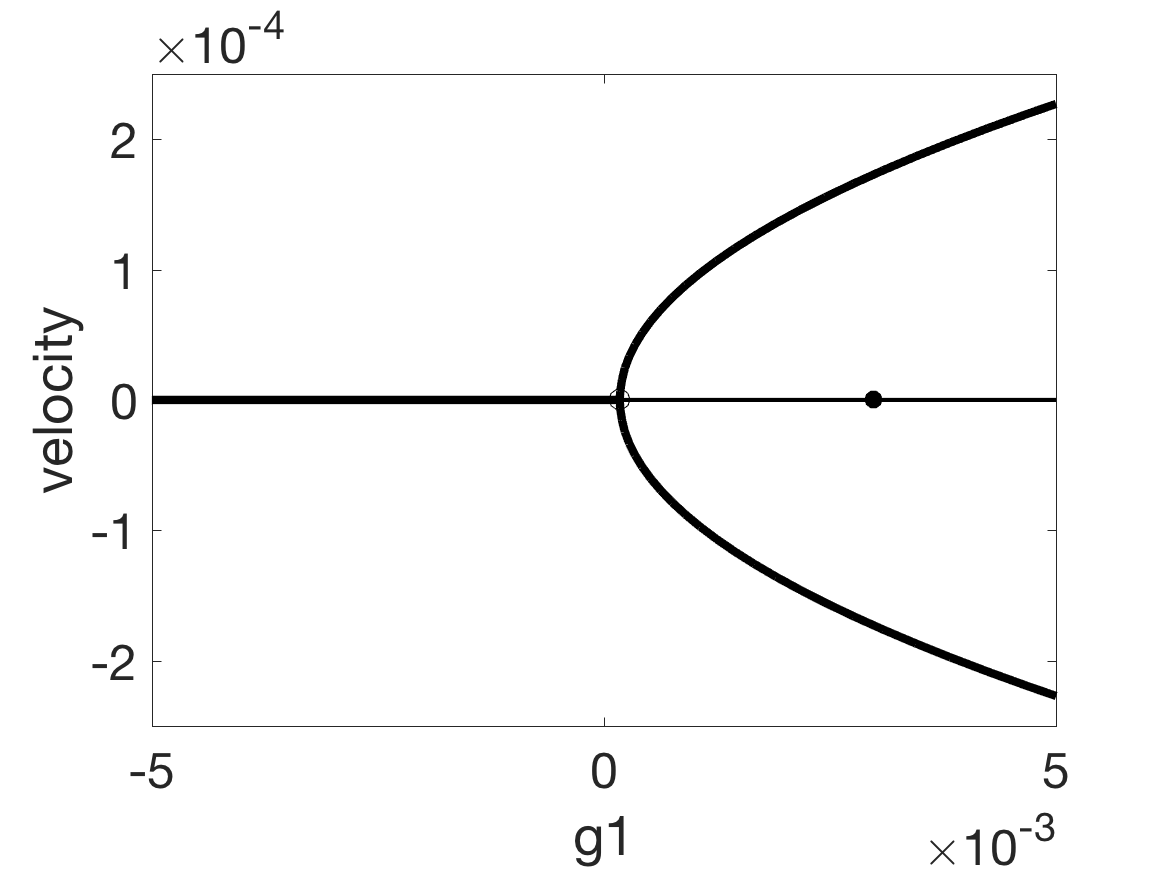}
& \includegraphics[width=0.3\textwidth]{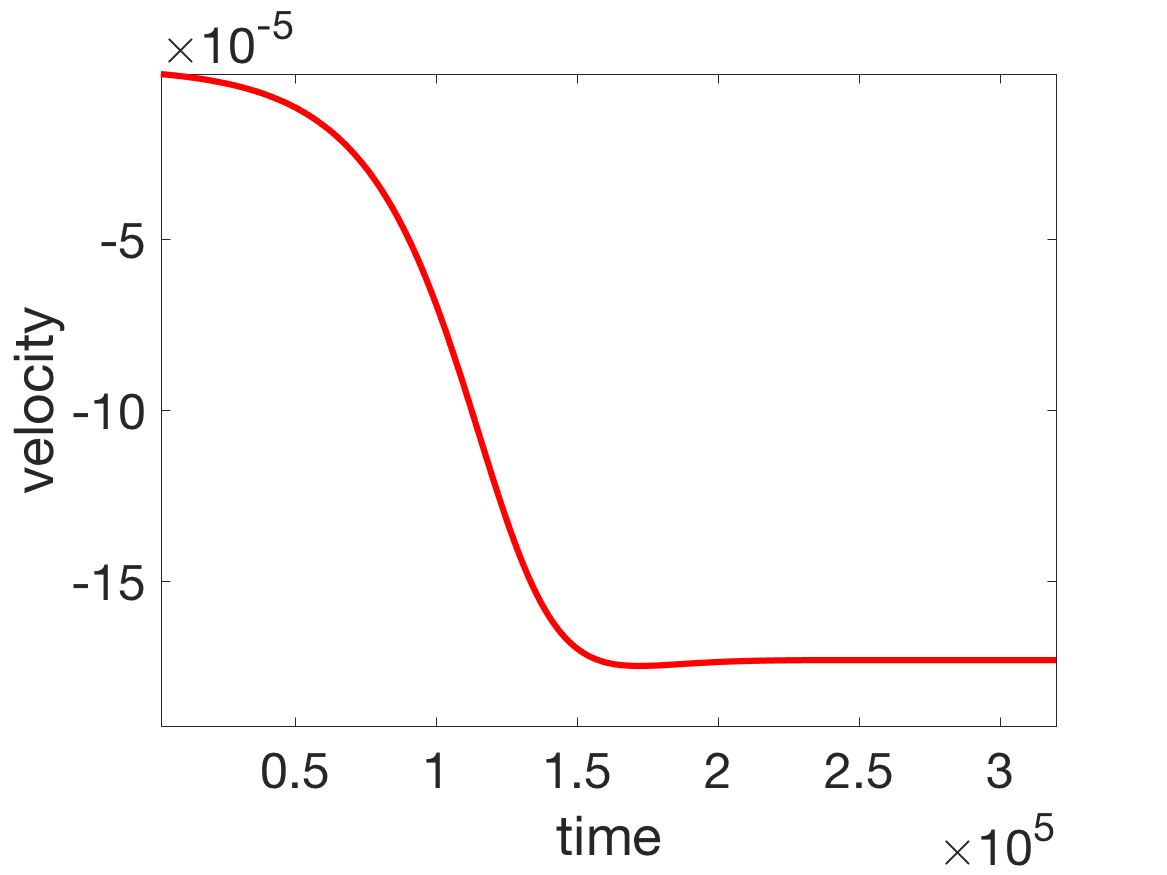}
&  \includegraphics[width=0.3\textwidth]{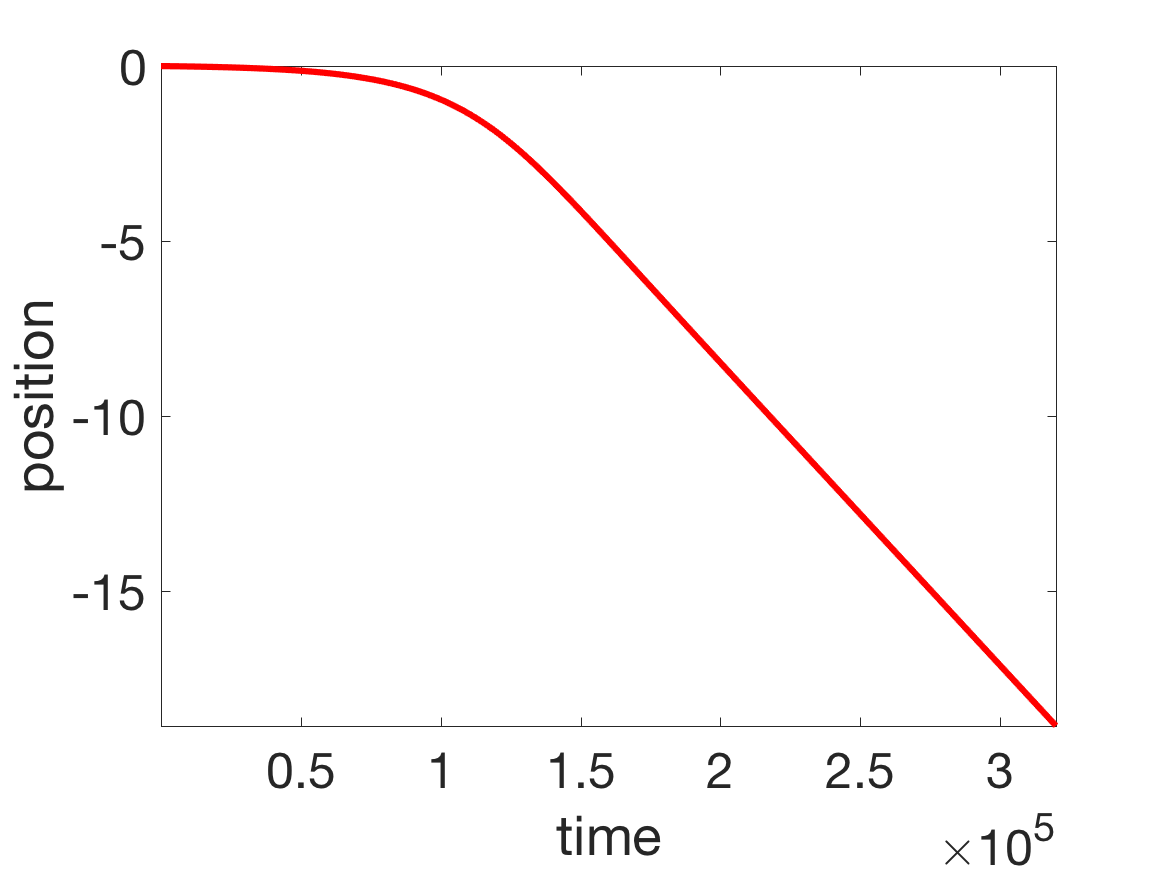}\\
(a) & (b) & (c) 
\end{tabular}
\caption{From region 1 to 6 of Figure~\ref{f:unfoldsketch11}, i.e., $\g_3\approx-0.074$, $\g_4\approx  -3.14$ with $\g_2\approx-0.02$, and \eqref{e:numfix}.
%
%
(a) Branches of fronts from numerical continuation (stable thick, unstable thin) connected at a pitchfork bifurcation at $\g_1\approx0.003$. 
(b,c) Plots of velocity and position from a simulation of a perturbation from the solution at the bullet in (a) with inset a magnification to highlight the  oscillatory convergence. Here $\g_1\approx 0.003$, and PDE parameters in \eqref{eq:three_component_system} are $\alpha\approx 0.34$, $\beta\approx-0.09$ and \eqref{e:numfix}.
}
\label{f:triple3}
\end{center}
\end{figure}

Finally, as noted in Remark~\ref{r:gam} for $\gamma=0$ there are no stable periodic traveling fronts with non-zero average speed, while for $\gamma\neq0$ these can be created. We plot an example in Figure~\ref{f:movper}.

\begin{figure}[tbp]
\begin{center}
\begin{tabular}{ccc}
\includegraphics[width=0.3\textwidth]{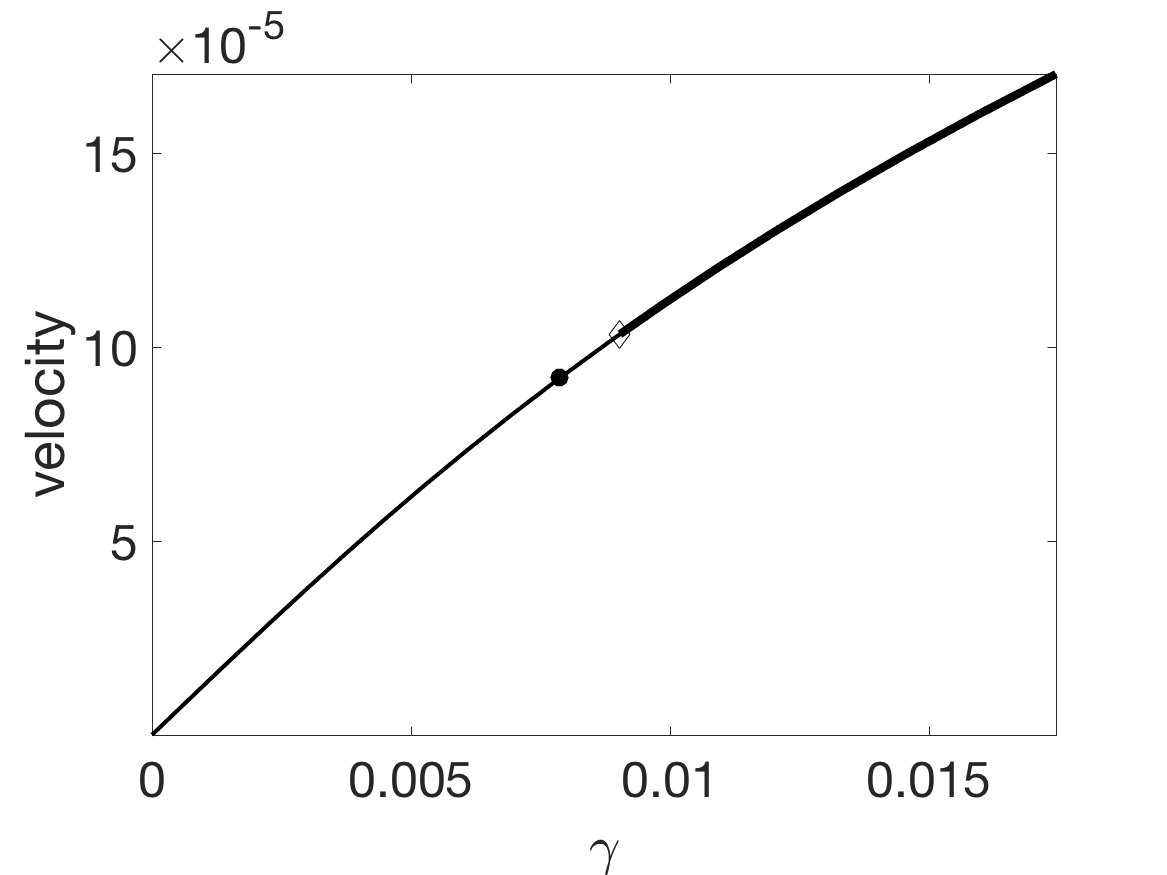}
&\includegraphics[width=0.3\textwidth]{vels-in-mu-pt12.png}
&\includegraphics[width=0.3\textwidth]{pos-in-mu-pt12.png}\\
(a) & (b) & (c)
\end{tabular}
\caption{Symmetry breaking with the parameter $\gamma$. (a) branch from numerical continuation starting at the rightmost point in Figure~\ref{f:triple-in-g2}(a). (b,c) plots of velocity and position from a simulation upon perturbing at  the point marked with a bullet in (a), and PDE parameters in \eqref{eq:three_component_system} are $\alpha\approx 0.48$, $\beta\approx-0.27$ and \eqref{e:numfix}, See also Figure~\ref{f:numintro}.
\label{f:movper}}
\end{center}
\end{figure}


\section{Conclusions and outlook}\label{s:conclusion}

We have demonstrated novel aspects in the rich dynamics of front solutions in the PDE \eqref{eq:three_component_system} by focusing on instabilities of stationary front solutions. Specifically, we gave a rigorous analysis revealing that the temporal evolution of the velocity of fronts is governed by a planar ODE, and we unfolded the bifurcation scenario of a Bogdanov-Takens point with symmetry for these. 
The main novelties of the present work consist of the rigorous argument for the existence of a second generalized eigenfunction for the operator arising from linearization around a stationary front, and in the effective method to compute the critical coefficients for the reduced system on the center manifold using solely information on the previously computed Evans function and existence condition for uniformly traveling fronts.

These results put us in a position to analyse the unfolding of the triple zero eigenvalue for front dynamics in the PDE \eqref{eq:three_component_system} with higher degeneracies: either the SBTD case or the imprint of a butterfly catastrophe in the SBTB case, see Definition~\ref{DEF}. These higher codimension problems require determining an additional center manifold coefficient, and also pose challenges on the level of the unfolding theory for ODEs, e.g.\ \cite{Khibnik}.

Equipped with the presented framework we expect to find Jordan chains of higher order upon addition of more slow components. That is, for the $(n+1)$-component system with the perturbed Allen-Cahn `fast' component $ U $ coupled to $n$ `slow' linear equations. In particular, such a 4-component system 
\begin{align} \nonumber 
  \left\{ 
 \begin{array}{rlrl}
                                         \partial_t U & = &  \eps^2 \partial_x^2 U & + \ U - U^3 - \eps \mathcal{G}(V_1, V_2, V_3) \, , \\[.2cm]
  \frac{\hat{\tau}_j}{\eps^2}   \,  \partial_t V_j  & = &  D_j^2  \partial_x^2 V_j & + \ U - V_j \, ,  \qquad j = 1, 2, 3 \, ,
 \end{array}
 \right.
\end{align}
would yield a Jordan block of length four and, hence, a three-dimensional reduced system on the center manifold (after factoring out translations). By appropriately changing the coupling of all components to imprint the desired singularity structure, a similar analysis as illustrated here could lead to a normal form of a chaotic system, and thus to one of the rare cases where chaos can be rigorously proved in the context of a nonlinear PDE.

\section*{Acknowledgements}
PvH thanks Leiden University for its hospitality.  JR notes this paper is a contribution to project M2 of the Collaborative Research Centre TRR 181 ``Energy Transfer in Atmosphere and Ocean" funded by the Deutsche Forschungsgemeinschaft (DFG, German Research Foundation) - project number 274762653. The authors also acknowledge that a crucial part of this manuscript was established during the first and second joint Australia-Japan workshop on dynamical systems with applications in life sciences.

\appendix


\section{Leading order form of eigenfunctions}\label{app:formal_computation}

In the proof of Proposition~\ref{proposition:jordan_block} and Remark~\ref{remark:first} we use leading order information on the eigenfunction and first generalized eigenfunction for the zero eigenvalue. The corresponding statements and proofs can be found here in Lemma~\ref{lemma:eigenfunctions} and Lemma~\ref{lemma:first_generalized_eigenfunction}. The section is completed by giving the leading order expressions for the second generalized eigenfunction in Lemma~\ref{lemma:second_generalized_eigenfunction}.


\begin{lemma}[Leading order of the eigenfunctions]\label{lemma:eigenfunctions}
The eigenfunctions $ \Phi_{\hlambda} $ belonging to the small eigenvalues from Proposition~\ref{proposition:evans_function} are to leading order given by
\begin{align*}
\left[
\begin{array}{c}
 0 \\[.2cm]
h_v(\hlambda) e^{ h_v(\hlambda) x} \\[.2cm]
h_w(\hlambda) e^{ h_w(\hlambda) x} 
\end{array}
\right]
\chi_{s-}(x)
+
\left[
\begin{array}{c}
 \frac{\frac12 \sqrt2}{\varepsilon} \, {\rm sech}^2\left[\frac{x}{\sqrt{2}\varepsilon}\right] \\[.2cm]
 h_v(\hlambda)  \\[.2cm]
 h_w(\hlambda) 
\end{array}
\right]
\chi_{f}(x)
+
\left[
\begin{array}{c}
 0 \\[.2cm]
 h_v(\hlambda)  e^{- h_v(\hlambda)  x} \\[.2cm]
 h_w(\hlambda) e^{-  h_w(\hlambda) x} 
\end{array}
\right]
\chi_{s+}(x) \, ,
\end{align*}
with
\[
 h_v(\hlambda) = \frac{1}{\sqrt{\htau \hlambda + 1}} \, , \quad  h_w(\hlambda) = \frac{1}{D\sqrt{\htheta \hlambda + 1}} \, .
\]
\end{lemma}

\bigskip

\begin{proof}
Using the notation $ \Phi_{\hlambda} = (u,v,w) $ for the eigenfunction corresponding to the eigenvalue $ \lambda = \eps^2 \hlambda $, the ODE arising from the eigenvalue problem \eqref{eq:evp} for small eigenvalues reads
\begin{align} \nonumber 
\left\{ 
 \begin{array}{rcl}
  \eps u' & = &  p \, , \\
  \eps p' & = & \eps^2 \hlambda u + (3 (u^{\rm h})^2 -1)u + \eps(\alpha v + \beta w) \, , \\
              v' & = &  q \, , \\
	      q' & = & (\htau \hlambda +1)v - u \, , \\
	      w' & = &  r \, , \\
	      r' & = & \frac{1}{D^2}(\htheta \hlambda +1)w - \frac{1}{D^2} u \, ,
  \end{array}
 \right.
\end{align}
In the language of slow-fast ODEs, this is the slow system with corresponding fast system given by
\begin{align} \nonumber 
\left\{ 
 \begin{array}{rcl}
   \dot{u} & = &  p \, , \\
   \dot{p} & = & \eps^2 \hlambda u + (3 (u^{\rm h})^2 -1)u + \eps(\alpha v + \beta w) \, , \\
   \dot{v} & = & \eps q \, , \\
   \dot{q} & = & \eps (\htau \hlambda +1)v - \eps u \, , \\
   \dot{w} & = & \eps r \, , \\
   \dot{r} & = & \frac{\eps}{D^2}(\htheta \hlambda +1)w - \frac{\eps}{D^2} u \, ,
  \end{array}
 \right.
\end{align}
where the dot denotes differentiation with respect to $ \xi = x/\eps$. In the regions $ I_{s^{\pm}} $ we will use a regular expansion of the eigenfunction in the slow system, while in the regions $ I_{f} $ we will use a regular expansion of the eigenfunction in the fast system. In the following we will use the following notation: Regular expansions of the amplitude will be denoted by $ u = u_0 + \eps u_1 + \eps^2 u_2 + \ldots$ and similarly for $ v, w $ and $ u^{\rm h} $. Furthermore, we add to the index `$f$' in the fast field and `$s{\pm}$' in the slow fields.\\

Before we demonstrate the calculations we would like to remark that we make use of the following observations from \cite{CDHR15} (which can already be found in \cite{DHK09}): We have that
\begin{align}
\label{eq:u_het_hot}
 u^{{\rm{h}}}_f = u^{\rm h}_{0,f} + \eps^2 u^{h}_{2,f} + \mathcal{O}(\eps^3) \, ,
\end{align}
that is, there is no first order correction of the stationary front in the fast field. Furthermore, we will need to use the value of the integral
\begin{align}\label{eq:integral_value}
 6 \int_{-\infty}^{\infty} u^{\rm h}_{0,f}(\xi)  u^{\rm h}_{2,f} (\xi){\rm sech}^4\left( \frac12\sqrt{2} \xi \right)d\xi = - 4 \left( \alpha + \frac{\beta}{D} \right) \, .
\end{align}
Equipped with these facts we will now recursively solve the perturbation hierarchy to construct an eigenfunction, i.e., a homoclinic to zero .\\ 

\underline{\it Fast field, $ \mathcal{O}(1) $}: We get for the $ u $-component
\[
 \ddot{u}_{0,f} = \left(3 (u^{\rm h}_{0,f}\right)^2 -1)u_{0,f} \, \\, \quad u_{0,f}(\xi) = C \,{\rm sech}\left( \frac12\sqrt{2} \xi \right)\, , C \in \mathbb{R} \, .
\]
while $\dot{v}_{0,f} = \dot{q}_{0,f} = \dot{w}_{0,f} = \dot{r}_{0,f} = 0 $. In order to compute the constant values assumed by these latter components we need to switch to the slow fields.\\

\underline{\it Slow fields, $ \mathcal{O}(1) $}: We have $ u_{0,s\pm}  = p_{0,s\pm} = 0$, while the equations for $ v,w $-components read
\[
 v_{0,s\pm}'' = (\htau \hlambda +1)v_{0,s\pm} \, , \quad w_{0,s\pm}'' = \frac{1}{D^2} (\htheta \hlambda +1)w_{0,s\pm} \, ,
\]
which are solved by exponentials. Using the information from the fast field that $ v, q, w, r $ are constant and matching slow and fast solutions gives that $ v_{0,s\pm}  = q_{0,s\pm} = w_{0,s\pm}  = r_{0,s\pm} = 0 $, therefore also $ v_{0,f}  = q_{0,f} = w_{0,f}  = r_{0,f} = 0 $.\\

\underline{\it Fast field, $ \mathcal{O}(\eps) $}: We get for the $ u $-component due to \eqref{eq:u_het_hot} and $ v_{0,f} = w_{0,f} = 0 $ again
\[
 \ddot{u}_{1,f} = \left(3 (u^{\rm h}_{0,f}\right)^2 -1)u_{1,f} \, \\, \quad u_{1,f}(\xi) = \cc \, {\rm sech}\left( \frac12\sqrt{2} \xi \right)\, , \cc \in \mathbb{R} \, .
\]
and we choose in this case $ \cc = 0 $ since, otherwise, this would simply add an $ \eps $ correction to $ C $ from the leading order. Furthermore, we have $ \dot{v}_{1,f} = \dot{w}_{1,f} = 0$ and
\begin{align}\label{eq:q_dot_r_dot_jump}
 \dot{q}_{1,f} = -u_{0,f} \, , \qquad  \dot{r}_{1,f} = - \frac{1}{D^2} u_{0,f} \, . 
\end{align}
Again, in order to compute the constant values assumed by these latter components we need to switch to the slow fields.\\

\underline{\it Slow fields, $ \mathcal{O}(\eps) $}: We have $ u_{1,s\pm}  = p_{1,s\pm} = 0$, while the equations for $ v,w $-components read
\[
 v_{1,s\pm}'' = (\htau \hlambda +1)v_{1,s\pm} \, , \quad w_{1,s\pm}'' = \frac{1}{D^2} (\htheta \hlambda +1)w_{1,s\pm} \, ,
\]
which is solved by
\[
 v_{1,s\pm}(x) = A_{\pm} e^{\mp \sqrt{\htau \hlambda +1} x} \, , \quad  w_{1,s\pm}(x) = B_{\pm} e^{\mp \frac{1}{D} \sqrt{\htheta \hlambda +1} x} \, ,
\]
where we already took into account that the eigenfunction components need to approach zero at the infinities. Again matching these solutions over the fast fields using $ v_{1,s-}(0) =      v_{1,s+}(0), w_{1,s-}(0) = w_{1,s+}(0) $ gives $ A_+ = A_-=: A, B_+ = B_-=: B $. Furthermore, matching the $ q,r $-components using \eqref{eq:q_dot_r_dot_jump} gives
\[
 q_{1,s+}(0) - q_{1,s-}(0) = \int_{-\infty}^{\infty} \dot{q}_{1,f}(\xi) d \xi = - \int_{-\infty}^{\infty} u_{0,f}(\xi) d \xi = - 2 \sqrt{2} C = -2 A \sqrt{\htau \hlambda +1} \, ,
\]
hence, $ A = \sqrt{2}C/\sqrt{\htau \hlambda +1} $. The analogous procedure for the $ r $-component gives $ B = \sqrt{2}C/(D\sqrt{\htheta \hlambda +1})  $. Hence, the values of the components in the fast fields are $ v_{1,f} = A , w_{1,f} = B $.\\

\underline{\it Fast field, $ \mathcal{O}(\eps^2) $}: We get for the $ u $-component due to \eqref{eq:u_het_hot} the equation
\[
 \ddot{u}_{2,f} = \left(3 (u^{\rm h}_{0,f}\right)^2 -1)u_{2,f} + 6 u^{\rm h}_{0,f} u^{\rm h}_{2,f} u_{0,f} + \alpha v_{1,f} + \beta w_{1,f} + \hlambda u_{0,f} \, ,
\]
for which we enforce the solvability condition

\begin{align*}
 &\underbrace{6 \int_{-\infty}^{\infty} u^{\rm h}_{0,f}(\xi)u^{\rm h}_{2,f}(\xi) u_{0,f}^2(\xi) d \xi}_{=-4C^2(\alpha+\beta/D)} \\ &\qquad+ \left(\alpha \frac{\sqrt{2}C}{\sqrt{\htau \hlambda +1}} + \beta \frac{\sqrt{2}C}{D\sqrt{\htheta \hlambda +1}}  \right) \underbrace{\int_{-\infty}^{\infty} u_{0,f}(\xi) d\xi}_{=2\sqrt{2}C} +  \hlambda  \underbrace{\int_{-\infty}^{\infty} u_{0,f}(\xi)^2 d \xi}_{=C^2 (4\sqrt{2}/3)} = 0  \, ,
\end{align*}
where we made use of \eqref{eq:integral_value}. Note that $ C $ drops out of the equation since, of course, eigenfunctions are only unique up to multiplication with a constant. We choose in the statement of the proposition
$
 C = \frac12 \sqrt2/\eps \, ,
$
since this scaling naturally arises when computing the eigenfunction for $ \lambda = 0 $ through differentiation of the the stationary front.
\end{proof}


\begin{lemma}[Leading order of first generalized eigenfunction]\label{lemma:first_generalized_eigenfunction}
Let the parameters be chosen such that \eqref{eq:double_zero} is satisfied, that is, that the zero eigenvalue has algebraic multiplicity two. Then there is a generalized eigenfunction $ \Psi $ which is to leading order given by
\begin{align}\label{eq:gef}
\begin{array}{c}
\left[
\begin{array}{c}
\Psi_u(x) \\[.2cm]
\Psi_v(x) \\[.2cm]
\Psi_w(x) 
\end{array}
\right]
= 
\left[
\begin{array}{c}
 \eps u_{s-}(x) \\[.3cm]
 v_{0,s-}(x) \\[.3cm]
 w_{0,s-}(x)
\end{array}
\right]
\chi_{s-}(x)
+
\left[
\begin{array}{c}
 \frac{\eps}{3 \sqrt{2}}  \\[.3cm]
- \frac{\htau}{2} \\[.3cm]
- \frac{\htheta}{2D} 
\end{array}
\right]
\chi_{f}(x)
+
\left[
\begin{array}{c}
 \eps u_{s+}(x)  \\[.3cm]
 v_{0,s+}(x) \\[.3cm]
 w_{0,s+}(x)
\end{array}
\right]
\chi_{s+}(x) \, ,
\end{array}
\end{align}
with
\[
v_{0,s\pm}(x) = - \frac12 \htau (1\pm x) e^{\mp x} \, , \quad w_{0,s\pm}(x) = - \frac12 \frac{\htheta}{D} \left(1 \pm \frac1D x\right) e^{\mp x/D} \, .
\]
and
\[
u_{s,\pm}(x) = - \frac{1}{2} \alpha v_{0,s\pm}(x) - \frac{1}{2} \beta w_{0,s\pm}(x) \, .
\]

\end{lemma}

\bigskip

\begin{proof} 
For notational simplicity we write $ (\Psi_u,\Psi_v,\Psi_w)  = (u,v,w) $. The ODE arising from the equation for the generalized eigenfunction reads 
\begin{align} \nonumber 
\left\{ 
 \begin{array}{rcl}
  \eps u' & = &  p \, , \\
  \eps p' & = & \eps^2 \Phi_u + (3 (u^{\rm h})^2 -1)u + \eps(\alpha v + \beta w) \, , \\
              v' & = &  q \, , \\
	      q' & = & v + \htau \Phi_v - u \, , \\
	      w' & = &  r \, , \\
	      r' & = & \frac{1}{D^2}w + \frac{\htheta}{D^2} \Phi_w - \frac{1}{D^2} u \, ,
  \end{array}
 \right.
\end{align}
recalling that $ \Phi $ is the eigenfunction for the zero eigenvalue and $ u^{\rm h} $ is the $ u $-component of the front solution \eqref{eq:front}. The corresponding fast system given by
\begin{align} \nonumber
\left\{ 
 \begin{array}{rcl}
   \dot{u} & = &  p \, , \\
   \dot{p} & = & \eps^2 \Phi_u + (3 (u^{\rm h})^2 -1)u + \eps(\alpha v + \beta w) \, , \\
   \dot{v} & = & \eps q \, , \\
   \dot{q} & = &  \eps v + \eps \htau \Phi_v- \eps u \, , \\
   \dot{w} & = & \eps r \, , \\
   \dot{r} & = & \frac{\eps}{D^2} w + \frac{\eps}{D^2} \htheta \Phi_w - \frac{\eps}{D^2} u  \, ,
  \end{array}
 \right.
\end{align}
where again the dot denotes differentiation with respect to $ \xi = x/\eps$.\\

\underline{\it Fast field, $ \mathcal{O}(1) $}: As before we get 
\[
 \ddot{u}_{0,f} = \left(3 (u^{\rm h}_{0,f}\right)^2 -1)u_{0,f} \,, \quad \dot{v}_{0,f} = \dot{q}_{0,f} = \dot{w}_{0,f} = \dot{r}_{0,f} = 0 \, ,
\]
We can choose $ u_{0,f} = 0 $ this time: on the one hand its value will not change the computations later on (it will appear as product with $ u^{\rm h}_{1,f} $ which is zero), on the other hand it is already part of the eigenfunction itself. In order to compute the constant values assumed by the other components we need to switch to the slow fields.\\

\underline{\it Slow fields, $ \mathcal{O}(1) $}: We have $ u_{0,s\pm}  = p_{0,s\pm} = 0$, while the equations for $ v,w $-components read
\[
 v_{0,s\pm}'' = v_{0,s\pm} + \htau e^{\mp x} \, , \quad w_{0,s\pm}'' = \frac{1}{D^2} w_{0,s\pm} + \frac{\htheta}{D^3} e^{\mp x/D} \, .
\]
We have that
\[
v_{0,s\pm}(x) = A_{\pm} e^{\mp x} \mp \frac12 \htau x e^{\mp x} \, , \quad w_{0,s\pm}(x) = B_{\pm} e^{\mp x/D} \mp \frac12 \frac{\htheta}{D^2} x e^{\mp x/D} \, , \qquad A_{\pm}, B_{\pm} \in \mathbb{R} \, .
\]
Matching with the information of the fast components $ v_{1,s-}(0) = v_{1,s+}(0), w_{1,s-}(0) = w_{1,s+}(0) $ gives $ A_+ = A_-=: A, B_+ = B_-=: B $ , while $ q_{1,s-}(0) = q_{1,s+}(0), r_{1,s-}(0) = r_{1,s+}(0) $ gives $ A = - \htau/2 , B = -  \htheta/(2D) $, so
\[
v_{0,s\pm}(x) = - \frac12 \htau (1\pm x) e^{\mp x} \, , \quad w_{0,s\pm}(x) = - \frac12 \frac{\htheta}{D} \left(1 \pm \frac1D x\right) e^{\mp x/D} \, .
\]
Hence, the values of the components in the fast fields are
\[
v_{0,f} = - \frac12 \htau \, , \quad w_{0,f} = - \frac12 \frac{\htheta}{D} \, .
\]

\underline{\it Fast field, $ \mathcal{O}(\eps) $}: We get for the $ u $-component due to \eqref{eq:u_het_hot} and the fact that 
\[
\Phi_u(\xi) = \left( \frac{1}{ \eps} \right) \, \frac12 \sqrt2 \,  {\rm sech}^2\left(\frac12 \sqrt2 \xi \right)
\]
the equation
\[
 \ddot{u}_{1,f} = \left(3 (u^{\rm h}_{0,f}\right)^2 -1)u_{1,f} +\frac12 \sqrt2 \,  {\rm sech}^2\left(\frac12 \sqrt2 \xi \right) \underbrace{+ \alpha v_{0,f} + \beta w_{0,f}}_{ = - \left(\frac12 \alpha \htau + \frac12 \frac{\beta}{D} \htheta \right)}  \, ,
\]
for which we get the solvability condition
\[
 \underbrace{ \frac12 \sqrt2 \int_{-\infty}^{\infty} {\rm sech}^4\left(\frac12 \sqrt2 \xi \right) d \xi}_{=4/3} - \left(\frac12 \alpha \htau + \frac12 \frac{\beta}{D} \htheta \right) \underbrace{\int_{-\infty}^{\infty}{\rm sech}^2\left(\frac12 \sqrt2 \xi \right)  d\xi}_{=2\sqrt{2}}  = 0  \, ,
\]
and, hence, the condition \eqref{eq:double_zero} which can also be written as $ {\calD}'(0) = 0 $. Rewriting this condition as
\[
- \left(\frac12 \alpha \htau + \frac12 \frac{\beta}{D} \htheta \right)  = -\frac{\sqrt{2}}{3}  \, ,
\]
and using the ansatz $ u_{1,f} = K \in \mathbb{R}$, we get
\[
 0 = \left[3 \, {\rm tanh}^2\left(\frac12 \sqrt2 \xi\right) -1\right]K +\frac12 \sqrt2 \,  {\rm sech}^2\left(\frac12 \sqrt2 \xi \right) - \frac{\sqrt{2}}{3}  \, ,
\]
which, by the identity $ {\rm sech}^2(z) = 1 - {\rm tanh}^2(z)  $ becomes
\[
 0 = \left[3 \, {\rm tanh}^2\left(\frac12 \sqrt2 \xi\right) -1\right]K + \frac{\sqrt2}{6} \, \left[1 - 3 \, {\rm tanh}^2\left(\frac12 \sqrt2 \xi\right)\right]  \, ,
\]
which gives $ K = \frac{1}{3 \sqrt{2}} $.
\end{proof}

The previous lemma was used in the proof of Proposition~\ref{proposition:jordan_block} for the existence of a second generalized eigenfunction. Here we give the formal computations that lead to first order expressions for it.


\begin{lemma}[Leading order of second generalized eigenfunction]\label{lemma:second_generalized_eigenfunction}
Let the parameters be chosen such that \eqref{eq:triple_zero} is satisfied, that is, that the zero eigenvalue has algebraic multiplicity three. Then there are two generalized eigenfunctions: $ \Psi $ as in Proposition~\ref{lemma:first_generalized_eigenfunction} and $ \widetilde{\Psi} $ which is to leading order given by
\begin{align*}
\left[
\begin{array}{c}
\widetilde{\Psi}_u(x) \\[.2cm]
\widetilde{\Psi}_v(x) \\[.2cm]
\widetilde{\Psi}_w(x) 
\end{array}
\right]
= 
\left[
\begin{array}{c}
 \eps \widetilde{u}_{s-}(x) \\[.3cm]
 \widetilde{v}_{0,s-}(x) \\[.3cm]
 \widetilde{w}_{0,s-}(x)
\end{array}
\right]
\chi_{s-}(x)
+
\left[
\begin{array}{c}
 \mathcal{O}(\eps^2) \\[.2cm]
 \frac{3\htau^2}{8} \\[.2cm]
 \frac{3\htheta^2}{8D} 
\end{array}
\right]
\chi_{f}(x)
+
\left[
\begin{array}{c}
 \eps \widetilde{u}_{s+}(x) \\[.3cm]
 \widetilde{v}_{0,s+}(x) \\[.3cm]
 \widetilde{w}_{0,s+}(x)
 \end{array}
\right]
\chi_{s+}(x)  + h.o.t. \, .
\end{align*}
with
\[
\widetilde{v}_{0,s\pm}(x) = \frac18 \htau^2 (x^2 \pm 3 x +3) e^{\mp x} \, , \quad \widetilde{w}_{0,s\pm}(x) = \frac{\htheta^2}{8D^3} (x^2 \pm 3D x + 3D^2) e^{\mp x/D}  \, .
\]
and
\[
\widetilde{u}_{s,\pm}(x) = - \frac{1}{2} \alpha \widetilde{v}_{0,s\pm}(x) - \frac{1}{2} \beta \widetilde{w}_{0,s\pm}(x) \, .
\]
\end{lemma}

\bigskip

\begin{proof} 
For notational simplicity we write $ (\widetilde{\Psi}_u,\widetilde{\Psi}_v,\widetilde{\Psi}_w)  = (u,v,w) $. The ODE arising from the equation for the second generalized eigenfunction reads 
\begin{align} \nonumber 
\left\{ 
 \begin{array}{rcl}
  \eps u' & = &  p \, , \\
  \eps p' & = & \eps^2 \Psi_u + (3 (u^{\rm h})^2 -1)u + \eps(\alpha v + \beta w) \, , \\
              v' & = &  q \, , \\
	      q' & = & v + \htau \Psi_v - u \, , \\
	      w' & = &  r \, , \\
	      r' & = & \frac{1}{D^2}w + \frac{\htheta}{D^2} \Psi_w - \frac{1}{D^2} u \, ,
  \end{array}
 \right.
\end{align}
recalling that $ \Psi $ is the first generalized eigenfunction for the zero eigenvalue \eqref{eq:gef} and $ u^{\rm h} $ is the $ u $-component of the front solution \eqref{eq:front}. The corresponding fast system given by
\begin{align} \nonumber  
\left\{ 
 \begin{array}{rcl}
   \dot{u} & = &  p \, , \\
   \dot{p} & = & \eps^2 \Psi_u + (3 (u^{\rm h})^2 -1)u + \eps(\alpha v + \beta w) \, , \\
   \dot{v} & = & \eps q \, , \\
   \dot{q} & = &  \eps v + \eps \htau \Psi_v- \eps u \, , \\
   \dot{w} & = & \eps r \, , \\
   \dot{r} & = & \frac{\eps}{D^2} w + \frac{\eps}{D^2} \htheta \Psi_w - \frac{\eps}{D^2} u  \, ,
  \end{array}
 \right.
\end{align}
where again the dot denotes differentiation with respect to $ \xi = x/\eps$.\\

\underline{\it Fast field, $ \mathcal{O}(1) $}: Exactly as before we get 
\[
 \ddot{\widetilde{u}}_{0,f} = \left(3 (u^{\rm h}_{0,f}\right)^2 -1)\widetilde{u}_{0,f} \,, \quad \dot{\widetilde v}_{0,f} = \dot{\widetilde q}_{0,f} = \dot{\widetilde w}_{0,f} = \dot{\widetilde r}_{0,f} = 0 \, ,
\]
and once again we can choose $ \widetilde{u}_{0,f} = 0 $, and  switch to the slow fields to determine the constant values the remaining components assume.\\

\underline{\it Slow fields, $ \mathcal{O}(1) $}: We have $ \widetilde{u}_{0,s\pm}  = p_{0,s\pm} = 0$, while the equations for $ v,w $-components read
\begin{align}\label{eq:inhom_ODE}
 \widetilde{v}_{0,s\pm}''- \widetilde{v}_{0,s\pm} =  - \frac12 \htau^2 (1\pm x) e^{\mp x} \, , \quad D^2 \widetilde{w}_{0,s\pm}'' - \frac{1}{D^2} \widetilde{w}_{0,s\pm} =  - \frac12 \frac{\htheta^2}{D} \left(1 \pm \frac1D x\right) e^{\mp x/D} \, . 
\end{align}
We have that
\[
\widetilde{v}_{0,s\pm}(x) = A e^{\mp x} + \frac18 \htau^2 (x^2 \pm 3 x ) e^{\mp x} \, , \quad \widetilde{w}_{0,s\pm}(x) = B e^{\mp x/D} + \frac{\htheta^2}{8D^3} (x^2 \pm 3D x) e^{\mp x/D} \, ,
\]
with $A, B \in \mathbb{R}$ and 
where we already used the matching with the information of the fast components 
$$ \widetilde{v}_{1,s-}(0) = \widetilde{v}_{1,s+}(0), \widetilde{w}_{1,s-}(0) = \widetilde{w}_{1,s+}(0) \, .$$
Furthermore, $ q_{1,s-}(0) = q_{1,s+}(0), r_{1,s-}(0) = r_{1,s+}(0) $ gives
\[
\widetilde{v}_{0,s\pm}(x) = \frac18 \htau^2 (x^2 \pm 3 x +3) e^{\mp x} \, , \quad \widetilde{w}_{0,s\pm}(x) = \frac{\htheta^2}{8D^3} (x^2 \pm 3D x + 3D^2) e^{\mp x/D}  \, .
\]
Hence, the values of the components in the fast fields are
\begin{align}\label{eq:inhom_ODE_sol}
\widetilde{v}_{0,f} = \frac38 \htau^2 \, , \quad \widetilde{w}_{0,f} = \frac38 \frac{\htheta^2}{D} \, . 
\end{align}

\underline{\it Fast field, $ \mathcal{O}(\eps) $}: We get for the $ u $-component due to \eqref{eq:u_het_hot} the equation
\[
 \ddot{\widetilde{u}}_{1,f} = \left(3 (u^{\rm h}_{0,f}\right)^2 -1)\widetilde{u}_{1,f} + \alpha \widetilde{v}_{0,f} + \beta \widetilde{w}_{0,f} \, ,
\]
for which we get the solvability condition
\[
 \frac38 \left(\alpha \htau^2 + \frac{\beta}{D} \htheta^2 \right) \underbrace{\int_{-\infty}^{\infty}{\rm sech}^2\left(\frac12 \sqrt2 \xi \right)  d\xi}_{=2\sqrt{2}}  = 0  \, ,
\]
and, hence, the triple zero eigenvalue condition \eqref{eq:triple_zero}, which can also be written as $ {\calD}''(0) = 0 $. Since by this condition, we again recover
\[
 \ddot{\widetilde{u}}_{1,f} = \left(3 (u^{\rm h}_{0,f}\right)^2 -1)\widetilde{u}_{1,f} \, ,
\]
we choose with a similar argument as before $ \widetilde{u}_{1,f} = 0 $.
\end{proof}


\section{Proof of Lemma~\ref{L4} (Spectrum of the operator $ L_\eps $)}\label{app:evp_perturbation}
Introducing the notation
\begin{align} \nonumber   
\Phi =  \begin{pmatrix} \Phi_{u}\\ \Phi_{v, w} \end{pmatrix} \, , 
\end{align}
for the eigenfunction of the zero eigenvalue, we can write the corresponding eigenvalue problem as
\begin{align}  \nonumber   
\begin{pmatrix} L_\eps & \eps \A\\ \B & \C\end{pmatrix}\begin{pmatrix} \Phi_{u} \\ \Phi_{v, w} \end{pmatrix}
=
\begin{pmatrix} 0 \\ 0 \end{pmatrix} \, .
\end{align}
By solving the second equation for $ \Phi_{v, w} $ we get $ \Phi_{v, w} = -S^{-1} B \Phi_{u}$, and inserted into the first equation this gives
\begin{align}  \nonumber   
 L_\eps \Phi_u = \eps A S^{-1} B \Phi_u \, . 
\end{align}
Recalling that $ L_\eps $ has the form
\begin{align}  \nonumber 
 L_\eps = \varepsilon^2 \partial_x^2 + \left(1- 3u_0^{\rm h}\left(\frac{x}{\eps}\right)^2 \right) + \mathcal{O}(\eps^2) \, ,
\end{align}
we change to the fast variable to $ y = \frac{x}{\eps} $ and write 
\begin{align} \nonumber
 L_\eps =  L_0 + \eps^2 L_1 + \ldots\, ,
\end{align}
with
\begin{align} \nonumber
 L_0 = \partial_y^2 + \left(1- 3u_0^{\rm h}(y)^2 \right) \, .
\end{align}
Note that, since $ S^{-1} $ is a convolution operator with respect to $ x $, changing to $ y = \frac{x}{\eps} $ gives an additional factor of $ \eps $, so we write $ S^{-1} = \eps \overline{S}^{-1} $, where now $ \overline{S}^{-1} $ gives the convolution with respect to $ y $. Furthermore, we set
\begin{align} \nonumber
 \Phi_u(y) = \frac{1}{\eps} \phi_0(y) + \eps \Phi_1(y) + \ldots \, , 
\end{align}
and after plugging all these expanded quantities back into 
\begin{align} \nonumber
 L_\eps \Phi_u = \eps^2 A \overline{S}^{-1} B \Phi_u \, ,
\end{align}
we get the equation for $ \Phi_1 $ given by
\begin{align} \nonumber
 L_0 \Phi_1 = - L_1 \phi_0 +  A \overline{S}^{-1} B \phi_0 \, ,
\end{align}
which yields the solvability condition
\begin{align}\label{eq:L_1_phi_0_1}
 \langle L_1 \phi_0, \phi_0 \rangle = \langle A \overline{S}^{-1} B \phi_0, \phi_0 \rangle \, .
\end{align}
Equipped with this, we can now turn to the eigenvalue problem
\begin{align} \nonumber
 L_\eps \phi = \eps^2 \widetilde{\mu}_{\eps} \phi  \, .
\end{align}
Setting $ \phi = \frac{1}{\eps} \phi_0 + \eps \phi_1 + \ldots , \widetilde{\mu}_{\eps} = \widetilde{\mu}_{0} + \ldots $ (noting that $ \Phi_u $ and $ \phi $ must coincide in leading order, but might differ in the next orders), we get
\begin{align}  \nonumber
 L_0 \phi_1 = - L_1 \phi_0 + \widetilde{\mu}_{0}  \phi_0 \, ,
\end{align}
yielding the solvability condition
\begin{align}\label{eq:L_1_phi_0_2}
 \widetilde{\mu}_{0} = \frac{\langle L_1 \phi_0, \phi_0 \rangle }{\langle  \phi_0, \phi_0 \rangle} \, .
\end{align}
Combining \eqref{eq:L_1_phi_0_1} and \eqref{eq:L_1_phi_0_2} gives
\begin{align}  \nonumber 
 \widetilde{\mu}_{0} = \frac{\langle A \overline{S}^{-1} B \phi_0, \phi_0 \rangle }{\langle  \phi_0, \phi_0 \rangle} \, .
\end{align}
Finally, using that $ \overline{S}^{-1} = \frac{1}{\eps} S^{-1} $ and that $ \frac{1}{\eps} \phi_0\left(\frac{x}{\eps}\right) $ is a Dirac sequence, we get as claimed in the limit $ \eps \rightarrow 0$
\begin{align} \nonumber 
  \widetilde{\mu}_{0} = \frac{3\sqrt2}{2} \left(\alpha + \frac{\beta}{D} \right) \, .
\end{align}

\end{document}